\documentclass{amsart}

\usepackage{amssymb}
\usepackage{amsmath}
\usepackage{tikz}
\usepackage{subfigure}
\usepackage{enumitem}
\usepackage{graphicx}

\numberwithin{equation}{section}
\newtheorem{thm}{Theorem}[section]
\newtheorem{prop}[thm]{Proposition}
\newtheorem{defn}{Definition}[section]
\newtheorem{rem}{Remark}[section]
\newtheorem{cor}[thm]{Corollary}
\newtheorem{lem}[thm]{Lemma}
\newtheorem{exam}{Example}

\def\R{\mathbb{R}}
\def\B{\mathbb{B}}
\def\M{\mathcal {M}}
\def\N{\mathcal {N}}
\def\T{\mathcal {T}}
\def\H{\mathcal {H}}
\def\C{\mathcal {C}}
\def\U{\mathcal {U}}
\def\K{\mathcal {K}}

\def\gr{\text{gr }}
\def\dist{\text{dist}}
\def\co{\text{co }}
\def\ae{\text{a.e. }}
\def\Esc{\text{Esc }}
\def\proj{\text{proj}}
\def\bdry{\text{bdry }}
\def\epi{\text{epi }}
\def\dom{\text{dom }}
\def\ext{\text{ext}}
\def\cone{\text{cone }}
\def\iint{\text{int }}
\def\ri{\text{\em{r-int} }}
\def\rb{\text{\em{r-bdry} }}
\def\meas{\text{\em{meas}}}
\def\Klim{\text{\em{Klim}}}

\def\bd{\begin{displaystyle}}
\def\ed{\end{displaystyle}}

\begin{document}	

\title{\bf{Flow Invariance on stratified domains}}
\author{R. C. Barnard} 
\address{Department of Mathematics, RWTH Aachen University, Aachen, Germany }
\email{barnard@mathcces.rwth-aachen.de}
\thanks{ First author supported by the VIGRE grant at LSU, DMS-0739382}
\author{P. R. Wolenski}
\address{Department of Mathematics, Louisiana State University, Baton Rouge, Louisiana, USA}
\email{wolenski@math.lsu.edu}

\maketitle


{\footnotesize
\noindent
\textbf{Abstract.}  This paper studies conditions for invariance of dynamical systems on stratified domains as originally introduced by Bressan and Hong.  We establish Hamiltonian conditions for both weak and strong invariance of trajectories on systems with non-Lipschitz data. This is done via the identification of a new multifunction, the essential velocity multifunction.  Properties of this multifunction are investigated and used to establish the relevant invariance criteria.
\\[3pt]

}

\vskip.2in

\section{Introduction}
\label{sec:intro}
This paper is concerned with characterizing strong and weak invariance of a dynamical system with non-Lipschitz data in the form of a so-called stratified system; this form was introduced by Bressan and Hong \cite{BH07}.  Our dynamical system takes the form of a differential inclusion
\begin{align*}\tag*{$(DI)_{G}$}
\qquad&\begin{cases}
\dot x(t)\in G\bigl(x(t)\bigr)\;\ae t\in [0,T]  \\
x(0)=x.
\end{cases}
\end{align*}
Stratified systems will be defined precisely in the next section, but to speak summarily, these are systems endowed with a layered structure of subdomains, each having their own dynamics that not only allow, but encourage movement in low dimensional submanifolds.  One can envisage a stratified system as a collection of \lq\lq highways\rq\rq\ on which rapid and efficient movement can be exploited in relation to the surrounding \lq\lq countryside\rq\rq.  This contrasts with previous studies by Soravia \cite{Sor02} and Camilli and Siconolfi \cite{CamSci03} who studied problems with much weaker assumptions but requiring relevant trajectories being unable to remain in a set of discontinuity for nontrivial time intervals.

Our approach to characterizations of flow invariance will closely follow the development in \cite{WZ98}.  In fact, the main difference between that and the current development is in the characterization of Strong Invariance (SI) (a new characterization of weak invariance is also required since a new Hamiltonian is utilized).  Recall that for dynamics $G$ and a closed set $\C\subseteq\R^N$, $G$ is SI on $\C$ provided every trajectory $x(\cdot)$ of (DI)$_G$ with $x\in\C$ has $x(t)\in\C$ for all $t\in[0,T]$.  With a Lipschitz hypothesis on the dynamics, an infinitesimal characterization of SI was first proved by Clarke \cite{FHC75} in tangential form, and later by Krastanov \cite{Kra95} in normal form.  See also \cite{CLSW95, BOOK98}.  The Lipschitz hypothesis was invoked in \cite{WZ98} precisely to exploit this characterization.  The first result (to our knowledge) characterizing SI for non-Lipschitz systems is contained in \cite{DRW05}, where the dynamics are one-sided (dissipative) Lipschitz.  We prove a new characterization for SI in Theorem~\ref{SI} below for discontinuous dynamics that satisfy an {\it Euler arc property}.  Moreover, we obtain a growth estimate of the distance back to $\C$ along a trajectory, a result that plays an important role later in the paper.

It is not clear whether the dynamics that arise from a stratified system have the Euler arc property globally, and so Theorem~\ref{SI} is not applied directly.  However, each subsystem restricted to its subdomain satisfies the property, and this turns out to be sufficient to prove the characterization on the entire domain.  A major complication that needs to be overcome is how to handle trajectories that cross manifolds infinitely often in finite time, exhibiting a type of \lq\lq Zeno\rq\rq\ effect that often plagues hybrid system theory; Example~\ref{Zeno} in Section~\ref{sec:Examples} presents a means to construct these types of complicated trajectories.  Inherent in the stratified apparatus is a structural condition on the dynamics providing the environment from which such arcs can be approximated by more manageable ones.  One of the key properties of stratified trajectories is that they cannot move {\it immediately} into another subdomain of the same or lower dimension, thus giving a particular form to how the Zeno effect can occur.  The proofs of Lemma~\ref{two subdomains} and Theorem~\ref{WI thm} contain details on how to construct approximate trajectories and deal with Zeno-type arcs.

The stratified system in this paper has additional structure beyond what was originally proposed in \cite{BH07}.  Specifically, we require the closure of each subdomain to be {\it proximally smooth} and to be {\it relatively wedged}.  Loosely speaking, proximal smoothness of the subdomains says there are no \lq\lq inward\rq\rq\ corners, which in turn implies the graph of the normal cone has a desirable closure property.  The wedged assumption means that none of the \lq\lq outward\rq\rq\ corners are cusps, which is saying every normal cone on the boundary has a relative pointedness condition.  Equivalently, this means the dimension of the relative interior of the tangent {\it cone} at each point throughout the boundary is the same as the dimension of the subdomain.  The wedged assumption is {\it relative} in the sense that, in local coordinates, it is active in the same subspace as the embedded submanifold.  The relative pointedness condition appears to be a new concept and is fully explained and developed in the appendix.  These additional properties provide the framework to describe the relative boundary of the tangent cone, the understanding of which is crucial to identifying the essential velocity multifunction $G^\sharp$.  The values of $G^\sharp$ are exactly those velocities that can be realized by trajectories of the system; see Proposition~\ref{prop of G Sharp} below.

Having identified $G^{\sharp}$ as the multifunction to characterize strong invariance, we must then assure that it is not too small to characterize weak invariance.  This turns out to be more of an issue than perhaps it first appears, and is confronted in Section \ref{WI stratified}.

The paper is organized as follows.  We begin in Section \ref{sec:introstrat} by describing the stratified formalism, and proceed in Section \ref{More Prelim} to provide further background and explain the relevant terms.  We relegate to the appendix the more complicated nonsmooth analysis involved in describing normal and tangent cones, and describe in detail there the \lq\lq relative wedged\rq\rq\ condition.  Section \ref{sec:SIsuff} is devoted to proving a new characterization of strong invariance for possibly discontinuous dynamics.  Strong invariance is characterized for stratified systems in Section \ref{sec:stratSI}, where the essential velocity multifunction $G^{\sharp}$ is defined.  The weak invariance characterization using $G^{\sharp}$ is given in Section \ref{WI stratified}.  Two examples are provided in Section \ref{sec:Examples}, the first providing a means to construct Zeno-type arcs, and the second constructing a type of Zeno arc with a somewhat surprising property.  Finally, Section \ref{sec:Conclusion} summarizes our approach and describes and suggests ongoing future research directions.

\section{Stratified systems}
\label{sec:introstrat}
We now describe a stratified dynamical system that was introduced by Bressan and Hong \cite{BH07}.   The closure of a set $\M$ is written as $\overline\M$.

\subsection{Stratified domains}
\label{subsec1.1}
We are given a finite collection of smooth (meaning at least $C^2$) manifolds $\{\M_1,\dots,\M_M\}$ embedded in $\R^N$ such that
\begin{itemize}
\item $\mathbb{R}^N=\bigcup_{i=1}^M \M_i$ and $\M_i\cap\M_j=\emptyset$ when $i\not= j$;
\item  If $\M_j\cap\overline{\M}_i\neq\emptyset$, then $\M_j\subset\overline{\M}_i$;
\item Each $\overline{\M}_i$ is proximally smooth of radius $\delta$;
\item Each $\overline \M_i$ is relatively wedged.
\end{itemize}
The last two conditions were not included in \cite{BH07}, but will be required in our study.  Their precise meaning and importance will be fully explained in Section~\ref{More Prelim}.  Such a collection is called a (proximally smooth and wedged) \textit{stratified domain}, and its components $\M_i$ are called stratified subdomains.  The dimension of $\M_i$ is designated by $d_i$.

\smallskip
\noindent
\begin{exam}\label{counterexample}  {\rm A very simple example is depicted in Figure~\ref{SimpleExample} with $N=2$ and $M=5$.  Here, the only two dimensional manifolds are the upper and lower open half spaces ($\M_1$ and $\M_2$), the only one dimensional ones are the positive and negative $x_1$-axes ($\M_3$ and $\M_4$), and the only zero dimensional manifold is the origin ($\M_5$).

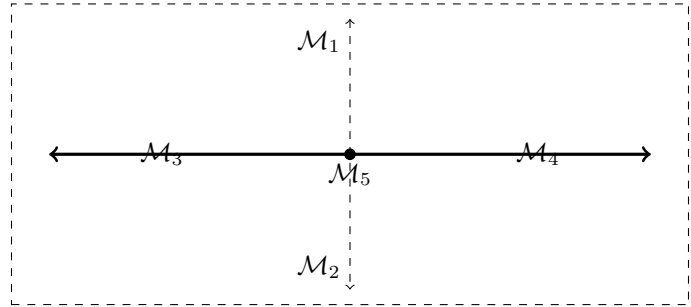
\begin{figure}[h]
\centering
\begin{tikzpicture}
\draw[dashed, thin] (-4.5,-2) -- (-4.5,2) -- (4.5,2) -- (4.5,-2) -- cycle;
\draw[<->,very thick] (-4,0) -- (4,0);
\draw[<->,dashed] (0,1.8) -- (0,-1.8);
\draw[fill] (0,0) circle(2pt);
\path (0,1.5) node [left] {$\M_1$};
\path (0,-1.5) node [left] {$\M_2$};
\path (-2.5,0) node {$\M_3$};
\path (2.5,0) node {$\M_4$};
\path (0,0) node [below] {$\M_5$};
\end{tikzpicture}\\
\caption{A stratified domain}
\label{SimpleExample}
\end{figure}
}
\end{exam}

\subsection{Stratified dynamics}

We next state dynamic hypotheses imposed on each of the subdomains.  Let $\M\subseteq \R^N$ be an embedded manifold, and denote by $\T_{\M}(x)$ the usual {\it tangent space} of $\M$ at $x\in\M$.  Suppose $\Gamma:\M\rightrightarrows\R^N$
is a multifunction.  The following is a collection of Standard Hypotheses commonly imposed in differential inclusion theory:
\[
{(SH)}
\quad
\begin{cases}
{\text{(i)}}\; &\forall\;x\in \M,\, \Gamma(x) \text{ is a nonempty, convex, compact set }\\
&\qquad \text{satisfying $\Gamma(x)\subseteq \T_{\M}(x)$; } \\
{\text{(ii)}} &\text{The graph }\gr \Gamma:=\bigl\{(x,v):v\in \Gamma(x)\bigr\}\text{ is a closed set relative }\\
&\qquad \text{to }\M\times\R^N; \\
{\text{(iii)}} &\exists\;r>0 \text{ so that }\max\{|v|:v\in \Gamma(x)\}\leq r(1+|x|).
\end{cases}
\]
Associated with $(\M,\Gamma)$ is the differential inclusion
\begin{equation}\tag*{(DI)$_{\Gamma}$}
\begin{cases}
\dot x(t)\in \Gamma\bigl(x(t)\bigr)\quad\ae t\in[0,T) \\
x(0)=x,
\end{cases}
\end{equation}
which, with $x\in\M$ and under the assumptions in (SH), has at least one solution $x(\cdot)$ defined on a nontrivial interval $[0,T)$.  Note the tangent space inclusion in (SH)(i) implies all solutions of (DI)$_\Gamma$ remain in $\M$.  If $T=+\infty$ or $x(t)$ approaches $\overline\M\backslash \M$ as $t\nearrow T$, then $T$ is called the {\it escape time} of $x(\cdot)$ from $\M$ and is denoted by $\Esc \bigl(x(\cdot),\M,\Gamma\bigr)$.  Every trajectory can be extended to an interval of maximal length, and thus if not otherwise stated, we assume all trajectories are defined on the maximal interval of existence.

A further hypothesis often invoked in standard differential inclusion theory is a strengthening of (SH)(ii) by requiring a Lipschitz property on bounded subsets of $\M$ with respect to the Hausdorff metric.  This means for each $r>0$, there exists a constant $k_r>0$ so that
\begin{itemize}
\item[(iv)]
\quad $ x,\,y\in\M\cap r\,\B\quad\Rightarrow\quad\dist_{\H}\bigl(\Gamma(x),\Gamma(y)\bigr)\leq k_r\,\|x-y\|,$
\end{itemize}
where $r\,\B$ denotes the ball centered at the origin of radius $r$, and $\dist_{\H}$ is the usual Hausdorff distance between compact sets.  Condition (iv) is equivalent to
\[
\bigl|h_{\Gamma}(x,\zeta)-h_{\Gamma}(y,\zeta)\bigr|\leq
k_r\,\|\zeta\|\,\|x-y\|\quad\forall x,\,y\in\M\cap r\,\B,\,\zeta\in\R^n,
\]
where $h_{\Gamma}:\R^N\times\R^N\to\R$ is the (minimized) Hamiltonian associated with $\Gamma$ given by
\[
h_{\Gamma}(x,\zeta)=\inf_{v\in\Gamma(x)}\langle v,\zeta\rangle.
\]
A manifold and multifunction $(\M,\Gamma)$ for which (i)-(iv) hold is said to satisfy (SH)$_+$.  Notice that the Lipschitz condition holds throughout any {\it bounded} subset of $\M$, and hence such a multifunction $\Gamma$ can be extended to the closure $\overline \M$ while maintaining the Lipschitz property over $\overline{\M}$.  We denote this extension by $\overline \Gamma$.

Now, a stratified domain $\{\M_1,\dots,\M_M\}$ is given, and associated to each manifold $\M_i$ is a given multifunction $F_i:\M_i\rightrightarrows \R^N$.  Each $(\M_i,F_i)$ is assumed to satisfy (SH)$_+$.  Since there are only finitely many objects, we can choose the same constants in (i)-(iv) for all of them.   Moreover, due to (iii), all of the analysis is essentially local (i.e. takes place in an a priori bounded set), and so for bookkeeping purposes we refer to only one Lipschitz constant $k$ in (iv) and drop the subscript $r$.

The {\it basic velocity }multifunction $F:\R^n\rightrightarrows\R^n$ is defined by
\[
F(x)=F_i(x)\quad \text{whenever}\quad x\in\M_i.
\]
Observe that $F$ does not necessarily satisfy (SH)(ii) on $\R^N$, and hence no general existence theory nor desirable closure properties of its trajectories are available.  To circumvent the difficulties these theoretical shortcomings would entail, one may introduce the Filippov regularization $G:\R^n\rightrightarrows\R^n$ given by:
\[
G(x)=\bigcap_{\varepsilon >0}\overline{\text{co }}\bigcup\bigl\{F(y):\|y-x\|<\varepsilon\bigr\}.
\]
It can be easily shown that $G$ satisfies (SH), although now it is no longer the case that the Lipschitz condition (iv) will hold for $(\R^N,G)$.  By the nature of the stratification structure, one can easily derive the representation
\begin{equation}\label{G rep}
G(x)=\co\bigl\{\overline{F}_{i}(x):x\in\overline\M_i\bigr\}.
\end{equation}

\setcounter{exam}{0}
\begin{exam}(continued)  {\rm Suppose the following dynamic data, which is a simplification of the dynamics in Example 3 of \cite{BH07}, is associated with the stratified domain in Example~\ref{counterexample} (we abuse set notation here since only singleton values are involved):
\[
F_1\left(\begin{matrix}x_1\\x_2\end{matrix}\right)=\left(\begin{matrix}0\\-1\end{matrix}\right);
\;
F_2\left(\begin{matrix}x_1\\x_2\end{matrix}\right)=\left(\begin{matrix}0\\1\end{matrix}\right);
\;
F_3\left(\begin{matrix}x_1\\x_2\end{matrix}\right)=\left(\begin{matrix}1\\0\end{matrix}\right);
\;
F_4\left(\begin{matrix}x_1\\x_2\end{matrix}\right)=\left(\begin{matrix}-1\\0\end{matrix}\right).
\]
Note that $F_5$ need not be explicitly stated since $\T_{\M_5}(0)=\{0\}$, and consequently by (SC), $F_5$ must consist of only the zero vector.  This is the case for all zero dimensional manifolds.  Obviously $G$ differs from $F$ only on the $x_1$-axis, with
\[
G\left(\begin{matrix} x_1\\0\end{matrix}\right)=\begin{cases} \left\{\vec v:
 |v_1|+|v_2|\leq 1,\,v_1\geq 0\right\} &\text{if }x_1<0 \\
\left\{\vec v:
 |v_1|+|v_2|\leq 1\right\}
 &\text{if }x_1=0 \\
 \left\{\vec v:
 |v_1|+|v_2|\leq 1,\,v_1\leq 0\right\} &\text{if } x_1>0,
 \end{cases}
\]
where $\vec v=\left(\begin{matrix}v_1\\v_2\end{matrix}\right)$.  The arrows in Figure~\ref{SimpleExample2} depict the values of the basic velocities, and the shaded regions are the values of $G$ at points on the $x_1$ axis.}
\end{exam}

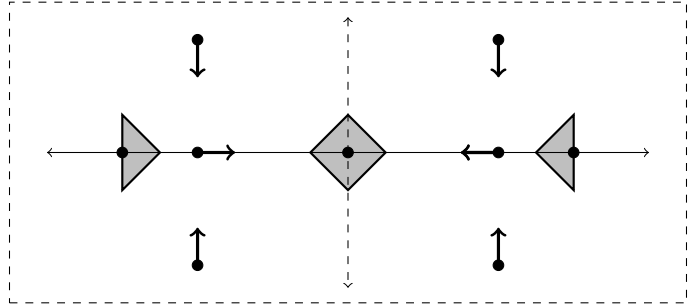
\begin{figure}
\centering
\begin{tikzpicture}
\draw[thick, fill, color=gray!50] (0,.5) -- (.5,0) -- (0,-.5) -- (-.5,0) -- cycle;
\draw[thick, fill, color=gray!50] (3,.5) -- (3,-.5) -- (2.5,0) -- cycle;
\draw[thick, fill, color=gray!50] (-3,.5) -- (-3,-.5) -- (-2.5,0) -- cycle;
\draw[thick] (0,.5) -- (.5,0) -- (0,-.5) -- (-.5,0) -- cycle;
\draw[thick] (3,.5) -- (3,-.5) -- (2.5,0) -- cycle;
\draw[thick] (-3,.5) -- (-3,-.5) -- (-2.5,0) -- cycle;
\draw[dashed, thin] (-4.5,-2) -- (-4.5,2) -- (4.5,2) -- (4.5,-2) -- cycle;
\draw[<->,thin] (-4,0) -- (4,0);
\draw[<->,dashed] (0,1.8) -- (0,-1.8);
\draw[fill] (0,0) circle(2pt);
\draw[fill] (-3,0) circle(2pt);
\draw[fill] (3,0) circle(2pt);
\draw[fill] (-2,0) circle(2pt);
\draw[->,very thick] (-2,0) -- (-1.5,0);
\draw[fill] (2,0) circle(2pt);
\draw[->,very thick] (2,0) -- (1.5,0);
\draw[fill] (-2,1.5) circle(2pt);
\draw[->,very thick] (-2,1.5) -- (-2,1);
\draw[fill] (2,1.5) circle(2pt);
\draw[->,very thick] (2,1.5) -- (2,1);
\draw[->,very thick] (2,0) -- (1.5,0);
\draw[fill] (-2,-1.5) circle(2pt);
\draw[->,very thick] (-2,-1.5) -- (-2,-1);
\draw[fill] (2,-1.5) circle(2pt);
\draw[->,very thick] (2,-1.5) -- (2,-1);
\end{tikzpicture}\\
\caption{Stratified dynamics}
\label{SimpleExample2}
\end{figure}

The multifunction $G$ is used as the dynamic data for the control problem.  For $x\in\R^N$, consider the differential inclusion:
\begin{equation}\tag*{(DI)$_{G}$}
\begin{cases}
\dot x(t)\in G\bigl(x(t)\bigr)\quad\ae t\in[0,T] \\
x(0)=x.
\end{cases}
\end{equation}

Again following \cite{BH07}, we assume that the stratified dynamics satisfy the Structural Condition
\begin{equation*}\tag{SC}
G(x)\cap \T_{\M_i}(x) = F_i(x)\quad\text{whenever}\quad x\in\M_i.
\end{equation*}
The condition (SC) plays a major role in this theory, and ensures that there does not exist a trajectory of $G$ that is not already implicit in $F$.  The following proposition makes this statement precise, and its proof is essentially contained in the argument at the bottom of page 319 in \cite{BH07}.  It is provided here for both completeness and clarity to assist in following the subsequent development.

\begin{prop}\label{G=F solutions}
Suppose $x(\cdot):[0,T]\to\R^N$ is a Lipschitz arc.  Then the following are equivalent.
\begin{itemize}
\item[(a)] $x(\cdot)$ satisfies (DI)$_{G}$;
\item[(b)] $x(\cdot)$ satisfies
\begin{equation*}
\tag*{(DI)$_F$}
\begin{cases}
\dot x(t)\in F\bigl(x(t)\bigr)\quad\text{a.e. }\in[0,T] \\
x(0)=x;
\end{cases}
\end{equation*}
\item[(c)] For each $i$, $x(\cdot)$ satisfies $x(0)=x$ and
\begin{equation*}
\tag*{(DI)$_{F_i}$}
\dot x(t)\in F_i\bigl(x(t)\bigr)\quad\text{a.e. whenever } x(t)\in\M_i.
\end{equation*}
\end{itemize}
\end{prop}
\begin{proof}
It is clear that (b) and (c) are equivalent, and that these imply (a) since $F_i(x)\subseteq G(x)$ whenever $x\in\M_i$.  So assume (a) holds and we must show that (c) holds as well.

For each $i$, let $J_i:=\{t\in[0,T]:x(t)\in\M_i\}$, which is a Borel measurable set.  Let
\[
\tilde J_i:=\{t\in J_i:\dot x(t)\text { exists in }G\big(x(t)\bigr)\text{ and }t\text{ is a Lesbesgue point of }J_i\}.
\]
It is clear that $\tilde J_i$ has full measure in $J_i$.    If $t\in\tilde J_i$, then being a Lebesque point implies there exists a sequence $\{s_j\}$ so that $s_j\to t$ as $j\to\infty$ with $t\not=s_j\in\tilde J_i$ for all $j$.  In particular, $x(s_j)\in\M_i$ for all $j$, and therefore
\[
\dot x(t)
=\lim_{j\to\infty}\frac{x(s_j)-x(t)}{s_j-t}\in\T_{\M_i}\bigl(x(t)\bigr).
\]
We conclude by (a) and the assumption (SC) that
\[
\dot x(t)\in G\bigl(x(t)\bigr)\cap \T_{\M_i}\bigl(x(t)\bigr)=F_i\bigl(x(t)\bigr)\quad\forall t\in\tilde J_i,
\]
which proves (c).
\end{proof}

\section{More preliminaries}\label{More Prelim}

We now review the background in nonsmooth analysis and differential inclusion theory required in our analysis.  The relative pointedness concept appears to be new, and thus the main consequences of this condition necessitates elaboration and detailed proofs.  These will be provided in the appendix.

\subsection{Background in nonsmooth analysis}
Recall that a vector $\zeta$ is a proximal normal to a closed set $\C\subseteq\R^N$ at $c\in\C$ provided there exists $\sigma>0$ so that
\[
\langle \zeta,c'-c\rangle \leq \frac{\|\zeta\|}{2\sigma}\|c'-c\|^2\quad\forall\,c'\in \C.
\]
In such a case with $\|\zeta\|=1$,  we have $\C\cap\bigl\{c+\sigma\,\bigl[\zeta+\overline{\B}\bigr]\bigl\}=\{c\}$ and say that $\zeta$ is {\it realized} by $\sigma$.  The set of all proximal normals is a convex cone and is denoted by $\N_{\C}^P(c)$.   If $\theta:\R^N\to(-\infty,+\infty]$ is lower semicontinuous (lsc) and $x\in\dom \theta:=\{x\in\R^N:\theta(x)<\infty\}$, then the proximal subgradient $\partial_P \theta(x)$ is defined as those $\zeta\in\R^N$ satisfying $(\zeta,-1)\in\N^P_{\epi \theta}\bigl(x,\theta(x)\bigr)$, where $\epi \theta:=\{(x,r):r\geq \theta(x)\}$ is the epigraph of $\theta(\cdot)$.

A key assumption of a stratification is {\it proximal smoothness}.  Proximal smoothness was a term introduced in \cite{CSW95} but whose concept was broached much earlier.  Federer \cite{Fed59} apparently first introduced the idea under the name of \lq\lq sets with positive reach,\rq\rq\ and it has been subsequently and independently rediscovered by many authors since then.

A closed set $\C\subseteq\R^N$ is called {\it proximally smooth} of radius $\delta>0$ provided the distance function
\[
d_{\C}(x):=\inf_{c\in\C}\|c-x\|
\]
is differentiable on the open neighborhood $\C+\{\delta+\varepsilon\}\B$ of $\C$ for some $\varepsilon>0$.   The $\varepsilon$ is added to this definition to avoid small technical anomalies.  There are many equivalent statements to this property, and we content ourselves here with mentioning only one: $\C$ is proximally smooth if and only if $\N_{\C}^P(c)\not=\{0\}$ for all $c\in\bdry\C$, and all unit normals can be realized by the same constant $\sigma:=\delta$.  In this case, the proximal normal cone coincides with the Clarke normal cone $\N_{\C}(c)$, and so in discussing normal cones to proximally smooth sets, we can shorten the notation of the proximal normal cone to $\N_{\C}(c)$ and observe that $c|\!\!\!\rightrightarrows\N_{\C}(c)$ has closed graph.

With $\C$ proximally smooth, there is no need to stipulate which tangent cone is in use either, since the Bouligand and Clarke tangent cones coincide.  Recall the Bouligand {\it tangent cone} $\T_{\C}(c)$ at $c\in\C$ is defined as
\begin{equation}\label{defn of tangent}
\T_{\C}(c)=\bigl\{v:\liminf_{t\downarrow 0}\frac{d_{\C}(c+tv)}{t}=0\bigr\}
\end{equation}
and in the case of $\C$ proximally smooth (and so is the Clarke cone as well), equals the negative polar of $\N_{\C}(c)$:
\[
v\in\T_{\C}(c)\quad\Longleftrightarrow\quad \langle \zeta,v\rangle \leq 0\quad\forall \zeta\in \N_{\C}(c).
\]

If $\M$ is an embedded $C^2$ manifold, $\C:=\overline\M$, and $c\in\M$, then $\T_{\C}(c)$ agrees with the usual tangent {\it space} $\T_{\M}(c)$ to $\M$ at $c$ from differential geometry (see \cite[Proposition~1.9]{BOOK98}).  If in addition $\overline M$ is proximally smooth, then for each $x\in\overline\M$, the tangent  cone $\T_{\overline \M}(x)$ is closed and convex, and thus has a relative interior denoted by $\ri\T_{\overline\M}(x)$ (see \cite{RTRconvex70}).  Its relative boundary is defined as $\rb\T_{\overline\M}(x):=\T_{\overline\M}(x)\backslash\ri\T_{\overline\M}(x)$.

Another key assumption on the stratified data is each domain being relatively wedged.  A set $\C\subseteq\R^N$ is {\it wedged} (see \cite[p.166]{BOOK98}) if for every $x\in \bdry\C$, the (Clarke) normal cone is pointed.  If $\C=\overline\M$ is the closure of an embedded manifold $\M$, then $\C$ is relatively wedged means each such normal cone is {\it relatively pointed}.  The concepts of pointedness and relative pointedness are fully explained in the appendix.  In short, within the context of a given stratification, the new concept of $\N_{\overline\M_i}(x)$ being relatively pointed is equivalent to saying the dimension of $\ri\T_{\overline\M_i}(x)$ is $d_i$.  Its key contribution to our analysis is exposed in the next lemma, which is proven in the appendix.

\begin{lem}\label{rb of tangent cone}
If $x\in\overline\M_i\backslash\M_i$ and $v\in \rb\T_{\overline{\M}_i}(x)$, then there exists an index $j$ for which $\M_j\subseteq\overline\M_i$, $x\in\overline\M_j$, and $v\in\T_{\overline\M_j}(x)$.  Of course in this case, one has $d_j<d_i$.
\end{lem}

\subsection{Background in differential inclusion theory}
References for this section are \cite{FHC83,BOOK98}.

Dynamic optimization relies heavily on the compactness of trajectory theorem (\cite[Theorem 3.1.7]{FHC83},\cite[Theorem~4.1.11]{BOOK98}).  This important theorem essentially asserts (with state space $\M=\R^N$) that a sequence of arcs that are almost trajectories of a differential inclusion with data satisfying (SH) has a subsequence that converges uniformly to an actual trajectory.  

Suppose $\Gamma:\R^N\rightrightarrows\R^N$ satisfies (SH).  The concept of an Euler solution to (DI)$_{\Gamma}$ was introduced in \cite{BOOK98}, and can be described as follows.  Let $\pi$ be a partition of $[0,T]$, say $\pi=\{0=t_0<t_1\dots <t_{\ell}=T\}$.  Usually one takes $t_{n}= n\frac{T}{\ell}$, but this is not necessary.  The norm of the partition is $\bd\|\pi\|:=\max_{n=1,\dots,\ell}\{t_{n}-t_{n-1}\}\ed$.  Let $\gamma(\cdot):\R^N\to\R^N$ be {\it any} function that is a selection of $\Gamma$; i.e. $\gamma(x)\in \Gamma(x)$ for all $x$.  Given the initial condition $x_0=x$ in (DI)$_{\Gamma}$, define a sequence of node points recursively by
\begin{eqnarray*}
x_1&=&x_0 + (t_1-t_0)\gamma(x_0)\\
x_2&=&x_1 + (t_2-t_1)\gamma(x_1)\\
\vdots &\vdots& \qquad \vdots \\
x_\ell&=&x_{\ell-1} + (t_{\ell}-t_{\ell-1})\gamma(x_{\ell-1}).
\end{eqnarray*}
The Euler polygonal arc $x^{\pi}(\cdot)$ is the piecewise linear function defined on $[0,T]$ that linearly interpolates the above sequence.  That is,
\[
x^{\pi}(t)= x_n + (t-t_n)\gamma(x_n)\quad \text{whenever }t\in[t_n,t_{n+1}].
\]
The compactness of trajectories theorem implies that for a given selection $\gamma(\cdot)$, a sequence $\{x^{\pi_m}(\cdot)\}$ of Euler polygonal arcs converges uniformly to a solution $x(\cdot)$ of (DI)$_{\Gamma}$ as $\|\pi_m\|\to 0$ as $m\to\infty$.  Such an $x(\cdot)$ is called an Euler arc.

Suppose $\M$ is an embedded manifold with $\overline\M$ proximally smooth of radius $\delta$.  We define the extension $\Gamma_{\text{ext}}:\bigl\{\overline\M+\delta\B\bigr\}\rightrightarrows\R^N$ of $\overline\Gamma$ by setting
\begin{equation}\label{extension}
\Gamma_{\text{ext}}(x)=\overline\Gamma\bigl(\proj_{\overline\M}(x)\bigr)
\end{equation}
whenever $x\in\overline\M+\delta\B$.  As (essentially) noted in \cite[Remark~4.9]{CSW95}, if $(\M,\Gamma)$ satisfies (SH) (respectively (SH)$_+$), then $(\overline \M+\delta\B,\Gamma_{\ext})$ also satisfies (SH) (respectively, (SH)$_+$)  The Lipschitz constant for $\Gamma_{\ext}$ in (iv) could perhaps be larger, but for notational simplicity we still denote it by $k$.

The Euler arc concept can be naturally extended to multifunctions defined on manifolds in one of (at least) three ways. These include (1) utilizing a change of variables with a fixed atlas of local charts, (2) applying the known concept to the extension $\Gamma_{\ext}$ as defined in (\ref{extension}), or (3) taking a projection back to the manifold at each discrete time step.  These methods will be compared elsewhere, but here we opt for the latter.  Recall $\gamma(\cdot)$ is a selection of $\Gamma$ on $\M$.  The node points are now chosen as
\[
x_{n+1}=\proj_{\overline\M}\biggl[x_{n}+(t_{n+1}-t_{n})\gamma\bigl(x_{n}\bigr)\biggr],
\]
and the process terminates when $x_{n+1}\in{\overline\M}\backslash\M$.  The Euler polygonal arc $x^{\pi}(\cdot)$ is defined as before.   Although $x^{\pi}(\cdot)$ may leave $\M$, since
$d_{\overline\M}\bigl[x^{\pi}_{n}+(t_{n+1}-t_{n})\gamma\bigl(x_{n}\bigr)\bigr]\leq c\,\|\pi\|$ for some constant $c$ independent of $\pi$ or $n$, the limiting Euler arc will lie in $\M$ and be a trajectory of (DI)$_{\Gamma}$.

Suppose $\M$ is an embedded manifold and $(\M,\Gamma)$ satisfies (SH).  Consider the property
\begin{itemize}
\item[(v)]  Every solution $x(\cdot)$ of (DI)$_\Gamma$ is an Euler arc.
\end{itemize}
Specifically, (v) means that if $\dot x(t)\in\Gamma\bigl(x(t)\bigr) \;\ae t\in[0,T]$ with $x(0)=x$, then there there exists a selection $\gamma(\cdot)$ of $\Gamma$ and a sequence of partitions $\{\pi_m\}$ with $\|\pi_m\|\to 0$ so that the Euler polygonal arcs $x^{\pi_m}$ associated with $\gamma(\cdot)$ converge uniformly to $x(\cdot)$.  We call (v) the {\it Euler arc property}, and note that the Lipschitz assumption (iv) implies (v):  This was shown for $\M=\R^N$ in \cite[Theorem~4.3.7]{BOOK98}.  In fact, a selection $\gamma(\cdot)$ can be produced for which the Euler arc is unique (i.e. is independent of the sequence of partitions).  The proof of the extension of this result to manifolds is routine.  We shall use the Euler arc assumption (v) in Theorem~\ref{SI} below when we give a new result on strong invariance.

We need another important result from $\R^N$ to be adapted to manifolds, and this is known as the Filippov approximate trajectory theorem (see \cite[Theorem~3.1.6]{FHC83}).  With $\Gamma$ defined on $\R^N$ and satisfying (SH)$_+$, this theorem says the following:  Suppose $y(\cdot)$ is a Lipschitz arc defined on $[0,T]$.  Define
\begin{equation}\label{rho}
\rho_{[0,T]}\bigl(y(\cdot)\bigr)=\int_0^T \dist\biggl(\dot y(t),\Gamma\bigl(y(t)\bigr)\biggr)\,dt,
\end{equation}
where for purely aesthetic purposes we switch to the notation $\dist(v,\C):=d_{\C}(v)$.  The theorem concludes that there exists a trajectory $z(\cdot)$ to (DI)$_{\Gamma}$ with $z(0)=y(0)$ and for which
\[
\bigl\|y(\cdot)-z(\cdot)\bigr\|\leq e^{kT}\,\rho_{[0,T]}\bigl(y(\cdot)\bigr).
\]
The result is valid even without $\Gamma$ having convex values; however, assuming it does, then an additional property can be asserted.  Namely, if $y(\cdot)$ is $C^1$, then the  solution $z(\cdot)$ can be chosen $C^1$ also.  This is not often explicitly stated, but follows from the construction of $z(\cdot)$ in the proof of \cite[Theorem~3.1.6]{FHC83}.  We adapt this result to manifolds in the next proposition.

\begin{prop}\label{Filippov}  Suppose $\M$ is an embedded manifold with $\overline\M$ proximally smooth of radius $\delta$, and $(\M,\Gamma)$ satisfies (SH)$_+$.  Suppose $y(\cdot):[a,b]\to \overline\M+\delta\B$ is a Lipschitz arc with $y(a)\in\M$, and for $a<t\leq b$, define
\[
\rho_{[a,t]}\bigl(y(\cdot)\bigr)=\int_a^{t} \dist\biggl(\dot y(s),\Gamma\bigl(\proj_{\overline\M}\bigl(y(s)\bigr)\bigr)\biggr)\,ds.
\]
Then there exists a trajectory $z(\cdot)$ of $\Gamma$ with $z(0)=y(0)$ that is defined on the interval $[a,\bar t)$, where $\bar t:= \min\{b,\Esc\bigl(z(\cdot),\M,\Gamma\bigr)\}$, so that
\[
\sup_{s\in [a,t]}\bigl\|y(s)-z(s)\bigr\|\leq e^{k(t-a)}\,\rho_{[a,t]}\bigl(y(\cdot)\bigr).
\]
for all $t\in [a,\bar t)$.  If $y(\cdot)$ is $C^1$, then $z(\cdot)$ can be chosen $C^1$ as well.
\end{prop}
\begin{proof}
Consider $\Gamma_{\ext}$ defined in (\ref{extension}).  The usual result \cite[Theorem~3.1.6]{FHC83}) now applies and yields the desired estimate.  The fact that $z(\cdot)$ actually remains in $\M$ is due to $z(0)=y(0)\in\M$ and that $(\M,\Gamma_{\text{ext}})$ is strongly invariant (see  (\ref{SI with Lip}) below, or \cite{FHC75,CLSW95,BOOK98}).
\end{proof}

\begin{rem}\label{full Filippov}
{\rm If it is the case that
\[
\inf_{z\in \overline \M\backslash\M}\|y(t)-z\|> e^{k\,(t-a)}\,\rho_{[a,t]}\bigl(y(\cdot)\bigr),
\]
for all $t\in[a,b]$, then $\bar t=b$ in Proposition~\ref{Filippov}.
}
\end{rem}

The following result is not usually emphasized in differential inclusion theory, and so we include it here along with a short proof.

\begin{prop}\label{v=dot x}
Suppose $(\M,\Gamma)$ satisfies (SH)$_+$ and $v\in \Gamma(x)$.  Then there exists a $C^1$ trajectory $x(\cdot)$ of (DI)$_{\Gamma}$ for which $\dot x(0)=v$.
\end{prop}
\begin{proof}  We can assume $\M=\R^N$.  Define $y(t)=x+tv$ for $t\in[0,T]$, which is a $C^1$ arc that satisfies
\[
\rho_{[0,t]}\bigl(y(\cdot) \bigr)=\int_0^t \dist\bigl(v,\Gamma(x+sv)\bigr)\,ds\leq \frac{k\|v\|}{2}t^2.
\]
By Proposition~\ref{Filippov}, there exists a $C^1$ arc $x(\cdot)$ defined on $[0,T]$ satisfying
\begin{equation}\label{dot x(0)=v}
\|y(t)-x(t)\|\leq\sup_{s\in[0,t]}\bigl\|y(s)-x(s)\bigr\|\leq \frac{1}{2}\,k\,\|v\|\,e^{kt}\,t^2.
\end{equation}
We have
\[
\bigl\|v-\dot x(0)\bigr\|=\lim_{t\searrow 0}\left\|\frac{y(t)-x}{t} - \frac{x(t)-x}{t} \right\|
=\lim_{t\searrow 0}\left\|\frac{y(t)-x(t)}{t} \right\|=0
\]
by (\ref{dot x(0)=v}).  Therefore $\dot x(0)=v$.
\end{proof}

\section{Sufficient condition for strong invariance}
\label{sec:SIsuff}
Suppose $\M\subseteq\R^N$ is an embedded manifold, a multifunction $\Gamma:\M\rightrightarrows\R^N$ satisfies (SH), and $\C\subseteq\R^N$ is closed.  Then $(\M,\Gamma)$ is said to be strongly invariant on $\C$ provided every solution $x(\cdot)$ of (DI)$_{\Gamma}$ with $x\in\M\cap \C$ is such $x(t)\in \C$ for all $t\in\bigl[0,\Esc\bigl(x(\cdot),\M,\Gamma\bigr)\bigr)$.  It is convenient in this section to notate the (maximized) Hamiltonian $H_{\Gamma}:\R^N\times\R^N\to\R$ by
\[
H_{\Gamma}(x,\zeta)=-h_{\Gamma}(x,-\zeta)=\max_{\gamma\in\Gamma(x)}\langle\gamma,\zeta\rangle.
\]
The following theorem is a sufficient condition that generalizes one half of a well-known characterization of strong invariance in the case when $\M=\R^N$ and $\Gamma$ satisfies (SH)$_{+}$.  The characterization with Lipschitz data is the HJ inequality
\begin{equation}\label{SI with Lip}
H_{\Gamma}(x,\zeta)\leq 0\quad\forall x\in\C,\,\zeta\in N^P_{\C}(x),
\end{equation}
and its equivalence was first proved in a tangential form by Clarke \cite{FHC75} (see also \cite{CLSW95,BOOK98}) and later in the normal form (\ref{SI with Lip}) by Krastanov \cite{Kra95}.  Although Clarke's original proof in \cite{FHC75} has the same flavor as ours, it nevertheless relied heavily on the Lipschitz property.

Our result has interest beyond stratified systems, and gives a sufficient condition for strong invariance for any potentially non-Lipschitz system whose only trajectories are Euler arcs.  This result should be compared with the main result in \cite{DRW05}, where the structure of a dissipative-Lipschitz multifunction was exploited to provide the first (to our knowledge) characterization for strong invariance with non-Lipschitz dynamics.  There is a nontrivial intersection between these results, but also a substantial difference in that \cite{DRW05} assumes a structural statement directly on the dynamic data, whereas the stratification system relies on state-dependent assumptions.

For a closed set $\C\subseteq\R^N$, denote by
\[
\C(x)=\proj_{\C}(x):=\{c\in \C:\|x-c\|=d_{\C}(x)\}
\]
the set of closest elements in $\C$ to $x$.  We also will later use the notation $\proj(x,\C)$ for $\proj_{\C}(x)$.

\begin{thm}\label{SI}
Suppose $\M\subseteq\R^N$ is a bounded embedded manifold and $\Gamma:\M\rightrightarrows\R^N$ is a multifunction satisfying (SH) and the Euler arc assumption (v).  Suppose $\C\subseteq \R^N$ is closed, and assume there exists a constant $\kappa>0$ so that
\begin{equation}\label{SI<}
x\in\M,\,c\in \C(x)\quad\Rightarrow\quad
H_{\Gamma}(x,x-c)\leq \kappa\, d_\C^2(x).
\end{equation}
Then for $x\in\M$ and $x(\cdot)$ any trajectory of (DI)$_{\Gamma}$, we have
\begin{equation}\label{SIest}
d_{\C}\bigl(x(t)\bigr)\leq e^{\kappa t}\,d_{\C}\bigl(x\bigr)\quad\forall t\in[0,T).
\end{equation}
In particular, $(\M,\Gamma)$ is strongly invariant on $\C$.
\end{thm}
\begin{proof}
Let $x\in\M$, $0<t<T$, $\pi=\{0=t_0<t_1,\dots<t_{\ell}=t\}$ a partition of $[0,t]$, and $x^{\pi}(\cdot)$ a polygonal Euler arc associated with $\pi$. Thus $x^{\pi}(\cdot)$ is piecewise linear on $[0,t]$, and for each $n=0,1, \dots,\ell-1$ satisfies
\[
x^{\pi}_{n+1}=proj_{\overline\M}\biggl[x^{\pi}_n+(t_{n+1}-t_{n})\gamma_n\biggr]
\]
for some $\gamma_n\in\Gamma\bigl(x^{\pi}_n\bigr)$, where $x^{\pi}_n:=x^{\pi}(t_n)$.  Let $\|\Gamma\|:=\sup\{\|\gamma\|:\gamma\in \Gamma(x),\,x\in r\B\cap\M\}$ (with $r$ sufficiently large), and choose any $c_n\in \C(x^{\pi}_n)$. For each $n$, we have
\begin{eqnarray}
d^2_{\C}\bigl(x^{\pi}_{n+1}\bigr) &\leq&
\|x^{\pi}_{n+1}-c_n\|^2  \nonumber\\
&= & \|x^{\pi}_n-c_n\|^2 +2(t_{n+1}-t_n)\bigl\langle \gamma_n,x^{\pi}_n-c_n\bigr\rangle + (t_{n+1}-t_n)^2\|\gamma_n\|^2 \nonumber\\
&\leq& \bigl(1+2\,(t_{n+1}-t_n)\,\kappa\bigr)
d_{\C}^2\bigl(x^{\pi}_n\bigr)+\|\pi\|^2\,\|\Gamma\|^2
\label{SI Main<}\\
&\leq &
\prod_{n'=0}^{n}\bigl(1+2(t_{n'+1}-t_{n'})\kappa\bigr)
d_{\C}^2\bigl(x\bigr) \nonumber \\
&&\qquad\qquad\qquad + \left(\frac{1-\bigl(1+2\|\pi\|\,\kappa\bigr)^{n+1}}{2\,\kappa}  \right)\|\pi\|\,\|\Gamma\|^2, \nonumber
\end{eqnarray}
where (\ref{SI Main<}) holds by invoking the assumption (\ref{SI<}).  Now if $x(\cdot)$ is any solution to (DI)$_{\Gamma}$, then by the Euler arc assumption (v), there exists a sequence of partitions $\pi_m$ and associated Euler arcs $x^m(\cdot):=x^{\pi_m}(\cdot)$ so that $\|\pi_m\|\to 0$ and $\|x^m(\cdot)-x(\cdot)\|\to 0$ as $m\to \infty$.  Each $x^m(\cdot)$ satisfies the previous estimate, and hence (\ref{SIest}) holds by letting $m\to\infty$ and taking square roots.
\end{proof}

\begin{cor}\label{SI cor}  Consider a stratified system as introduced above.  Suppose $\C\subseteq \R^N$ is closed and assume that (\ref{SI<}) holds for $\M=\R^N$ and $\Gamma=G$.  For a fixed $i$, if $x(\cdot):[a,b]\to \overline\M_i$ is a trajectory of $G$ with $x(t)\in\M_i$ for all $t\in(a,b)$, then
\begin{equation*}
d_{\C}\bigl(x(b)\bigr)\leq e^{k(b-a)}d_{\C}\bigl(x(a)\bigr).
\end{equation*}
\end{cor}
\begin{proof}
For all small $\varepsilon>0$, the arc $x(\cdot)$ restricted to $[a+\varepsilon,b-\varepsilon]$ is contained in $\M_i$ and therefore is a trajectory of $F_i(\cdot)$.  Since $\bigl(\M_i,F_i(\cdot)\bigr)$ satisfies (SH)$_+$ it must also satisfy the Euler arc assumption (v).  Since $G=F_i$ when $x\in\M_i$, we have that (\ref{SI<}) holds for $(\M,\Gamma)=(\M_i,\Gamma_i)$.  It follows from Theorem~\ref{SI} that
\[
d_{\C}\bigl(x(a+\varepsilon)\bigr)\leq e^{k(b-a-2\varepsilon)}d_{\C}\bigl(x(b-\varepsilon)\bigr).
\]
Letting $\varepsilon\downarrow 0$ finishes the proof.
\end{proof}

\section{Strong invariance in stratified systems}
\label{sec:stratSI}
By attempting to characterize a value function as a solution to an HJ equation, it is not enough to merely have a sufficient condition for strong invariance as in Theorem~\ref{SI}, but rather one needs a full characterization.  This is closely related to the two objectives stated in the introduction, where the necessity relates to (Obj1) and sufficiency to (Obj2).  Theorem~\ref{SI} contains only a sufficient condition, and so we must seek a necessary one for stratified systems.  This is the point that the stratification assumptions play their greatest role.  Recall we are given a stratified domain $\{\M_1,\dots,\M_M\}$ along with stratified dynamics encapsulated in $G$.

\begin{defn}\label{defn of G sharp}
The {\em essential velocity} multifunction $G^{\sharp}:\R^N\rightrightarrows\R^N$ is defined by
\begin{equation}\label{def of G sharp}
G^{\sharp}(x)=\bigcup_{i} \bigl\{\overline F_i(x)\cap \T_{\overline\M_i}(x):x\in \overline\M_i\bigr\}.
\end{equation}
\end{defn}

The definition of $G^{\sharp}$ should be compared with the representation (\ref{G rep}) of $G$.  One has $G^{\sharp}$ lying \lq\lq between\rq\rq\ $F$ and $G$; that is, for each $x$, we have
\begin{equation}\label{containments}
F_i(x)\subseteq G^{\sharp}(x)\subseteq G(x)\quad \text{whenever } x\in\M_i.
\end{equation}
In general, $G^{\sharp}$ will not possess the desirable properties typically invoked in differential inclusion theory.  For example, although its values are compact, they are not necessarily convex (violating (SH)(i)); nor is its graph necessarily closed (violating (SH)(ii)).  Its designation as the essential velocity set is based on the following, and should be compared with Proposition~\ref{v=dot x} which assumed Lipschitz dynamics.

\begin{prop}\label{prop of G Sharp}
Suppose $x(\cdot)$ is a solution to (DI)$_G$ defined on $[0,T)$.  Then $\dot x(t)\in G^{\sharp}\bigl(x(t)\bigr)$ for almost all $t \in[0,T)$.  Conversely, if $v\in G^{\sharp}(x)$, then there exists $T>0$ and a $C^1$ solution $x(\cdot)$ to (DI)$_G$ with $\dot x(0)=v$.
\end{prop}
\begin{proof}
The first statement of the theorem follows immediately from (\ref{containments}) and Proposition~\ref{G=F solutions}.  As for the second statement, we separate the most important fact into the following lemma.  A somewhat more general version of this lemma appeared in \cite{CW96}; we offer a much simpler proof here by exploiting the stratified structure.

\begin{lem}\label{essential derivative}
Suppose $\M$ is an embedded manifold with $\overline\M$ proximally smooth of radius $\delta$, the multifunction $\Gamma:\M\rightrightarrows \R^N$ satisfies (SH)$_+$, and $x\in\overline\M$ with $\N_{\overline\M}(x)$ relatively pointed.  Then for any $v\in \overline\Gamma(x)\cap\ri\T_{\overline\M}(x)$, there exist $T>0$ and a $C^1$ trajectory $x(\cdot):[0,T]\to\M\cup\{x\}$ so that $x(0)=x$ and $\dot x(0)=v$.
\end{lem}
\begin{proof}
By considering local coordinates, we may assume without loss of generality that the dimension of $\M$ is $N$.  Recall from (\ref{extension}) the extension $\Gamma_{\text{ext}}$ of $\Gamma$, which satisfies (SH)$_+$ on $\overline\M+\delta\B$ and agrees with $\overline\Gamma$ on $\overline\M$.  By Proposition~\ref{v=dot x}, there exists $T>0$ and a $C^1$ trajectory $x(\cdot)$ of $\Gamma_{\text{ext}}$ with $x(0)=x$ and $\dot x(0)=v$.  We will show $x(t)\in\M$ for all small $t$ which would finish the proof of the lemma.

Since $v\in\iint\T_{\overline\M}(x)$, by (\ref{int of tangent cone}) there exists $\mu>0$ so that
\[
\langle \zeta,v\rangle\leq -\mu\,\|\zeta\|\quad \forall \zeta\in\N_{\overline\M}(x).
\]
Reduce $T$ if necessary so that whenever $0<s<T$, we have $\|\dot x(s)-v\|\leq \frac{\mu}{2}$.  Then for all $t\in(0,T]$ and all $\zeta\in\N_{\overline\M}(x)$, we have
\begin{eqnarray}
\left\langle\frac{x(t)-x}{t},\zeta\right\rangle
&=&\frac{1}{t}\int_0^t \langle\dot x(s),\zeta\rangle\,ds
\nonumber \\
&=& \langle v,\zeta\rangle
+\frac{1}{t}\int_0^t \langle\dot x(s)-v,\zeta\rangle\,ds
\label{trajectory inclusion}\\
&\leq& -\mu\|\zeta\|+\frac{\mu}{2}\|\zeta\|= -\frac{\mu}{2}\|\zeta\| \nonumber
\end{eqnarray}
Suppose there exists a sequence $t_n\searrow 0$ with $x(t_n)\notin \M$.  Let $x_n=\proj_{\overline\M}\bigl(x(t_n)\bigr)$ and
\[
\zeta_n =
\begin{cases}
\frac{x(t_n)-x_n}{\|x(t_n)-x_n\|} &\text{if  }x(t_n)\notin\overline{\M}  \\
\bar \zeta_n & \text{if  }x(t_n)\in\overline{\M}
\end{cases}
\]
where $\bar\zeta_n$ is any unit vector in $\N_{\overline{\M}}\bigl(x(t_n)\bigr)$.  Without loss of generality, we can assume $\zeta_n\to\zeta\in\N_{\overline\M}(x)$.  We have
\begin{eqnarray}
\left\langle\frac{x(t_n)-x}{t_n},\zeta\right\rangle
&=& \left\langle\frac{x(t_n)-x}{t_n},\zeta-\zeta_n\right\rangle
+\left\langle\frac{x(t_n)-x_n}{t_n},\zeta_n\right\rangle
+\left\langle\frac{x_n-x}{t_n},\zeta_n\right\rangle\nonumber\\
&\geq& -\|\zeta_n-\zeta\|\,\frac{\|x(t_n)-x\|}{t_n}
+0-\frac{\|x_n-x\|^2}{2\delta\,t_n}.\label{lower bound}
\end{eqnarray}
The justifications for the three lower bounds in (\ref{lower bound}) are, respectively, the Cauchy-Schwarz inequality, the nature of $\zeta_n$, and the facts that $x\in\overline\M$ and $\zeta_n\in\N_{\overline\M}(x_n)$ is realized by $\delta$.
Observe the obvious estimate
\begin{equation}\label{bound on x(t)-x}
\|x(t)-x\|\leq \int_0^t\|\dot x(s)\|\,ds\leq \|\Gamma\|\,t,
\end{equation}
where $\|\Gamma\|$ is an upper bound on the values of $\Gamma$ that can occur here, and
\begin{equation}\label{bound on x_n-x}
\|x_n-x\|\leq\|x(t_n)-x_n\|+\|x(t_n)-x\|\leq 2\,\|x(t_n)-x\|\leq 2\,\|\Gamma\|\,t_n,
\end{equation}
where the second inequality follows since $x_n=\proj_{\overline\M}\bigl(x(t_n)\bigr)$ and $x\in\overline\M$, and the last one from (\ref{bound on x(t)-x}).  Letting $n\to\infty$ in (\ref{lower bound}) and using the estimates obtained in (\ref{bound on x(t)-x}) and (\ref{bound on x_n-x}) leads to a contradiction of (\ref{trajectory inclusion}). The conclusion is $x(t)\in\M$ for all small $t$ as claimed.
\end{proof}

Now we return to the proof of Proposition~\ref{prop of G Sharp}.  Let $v\in G^{\sharp}(x)$, and let $i$ be such that $x\in \overline \M_i$ and $v\in \overline F_i(x)\cap\T_{\overline \M_i}(x)$.  If $v\in \ri\T_{\overline \M_i}(x)$, then the result follows by Lemma~\ref{essential derivative}.  If $v\notin \ri\T_{\overline \M_i}(x)$, then $v\in \rb\T_{\overline \M_i}(x)$ and hence by Lemma~\ref{rb of tangent cone}, there exists another subdomain $\M_j\subseteq \overline\M_i$ with $x\in\overline\M_j$, $v\in \T_{\overline \M_j}(x)$, and $d_j<d_i$.  We also claim that
\begin{equation}\label{claim}
v\in \overline F_j(x).
\end{equation}
Indeed, since $x\in\overline\M_j$, there exists a sequence $x_n\to x$ with $x_n\in\M_j\subseteq\overline\M_i$.  Let $v_n=\proj\bigl(v,\overline F_j(x_n)\bigr)$.  Clearly $v_n\to v$ since $\overline F_j(\cdot)$ is Lipschitz on $\overline\M_j$, and therefore $v\in\overline F_j(x)$ as claimed in (\ref{claim}).  We now have $v\in \overline F_j(x)\cap\T_{\overline \M_j}(x)$ and the argument just given can be repeated with $i$ replaced by $j$.  This can be repeated as necessary but must terminate since the dimension is decreasing at each step.  A stage is reached when $v$ lies in the relative interior of the tangent cone of the subdomain, at which point Lemma~\ref{essential derivative} can once again be invoked to complete the proof.
\end{proof}

\begin{cor}\label{T is lower}  Suppose $\C\subseteq\R^N$ is closed and $(\R^N,G)$ is strongly invariant on $\C$.  Then
\begin{equation}\label{HJ >}
h_{G^\sharp}(x,-\zeta)\geq 0\quad\forall x\in\C,\,\zeta\in\N^P_{\C}(x).
\end{equation}
\end{cor}
\begin{proof}
Let $x\in\C$ and $\zeta\in\N^P_{\C}(x)$ be such that $\|\zeta\|=1$ and $\zeta$ be realized by $\sigma>0$.  Fix $v\in G^{\sharp}(x)$.  By Proposition~\ref{prop of G Sharp}, there exists a $C^1$ solution to (DI)$_{G}$ defined on an interval $[0,T]$.  By the strong invariance property, we have $x(t)\in\C$ for all $t\in[0,T]$, and therefore
\[
\langle v,\zeta\rangle = \lim_{t\searrow 0}\;
\left\langle\frac{x(t)-x}{t},\zeta\right\rangle\leq
 \lim_{t\searrow 0} \;\frac{1}{2\,\sigma\,t}\bigl\|x(t)-x\bigr\|^2=0,
\]
where the inequality follows from $\zeta\in\N^P_{\C}(x)$ and $x(t)\in\C$, and the last equality from (\ref{bound on x(t)-x}).  Taking the sup over $v\in G^\sharp(x)$ yields (\ref{HJ >}).
\end{proof}

We now turn to the sufficiency of (\ref{HJ >}) for strong invariance.  The proof of this result is adapted from the clever and difficult argument used in the proof of \cite[Theorem 3]{BH07}.

\begin{thm}\label{SI for stratified}  Suppose $\C$ is closed and (\ref{HJ >}) holds.  Then
\begin{equation}\label{SI estimate}
d_{\C}\bigl(x(T)\bigr)\leq e^{k\,T}d_{\C}\bigl(x(0)\bigr).
\end{equation}
for any $x(\cdot)$ that is a solution of (DI)$_G$.  In particular, $(\R^N,G)$ is strongly invariant on $\C$.
\end{thm}
\begin{proof}  Let $\C\subseteq\R^N$ be closed and $x(\cdot)$ a solution of (DI)$_G$.  Recall Corollary~\ref{SI cor}, where the estimate (\ref{SI estimate}) was shown to hold provided $x(\cdot)$ restricted to the open interval $(0,T)$ was a trajectory that resided in only one subdomain.  The following lemma contains the key fact and its proof is perhaps the hardest analysis in the entire paper.

\begin{lem}\label{two subdomains}
Suppose $\M_i$ is a subdomain and $\M$ is a union of subdomains with $\M_i\subseteq\overline \M$.  Assume $\M$ has the property that every trajectory $x(\cdot)$ of (DI)$_{G}$ defined on $[a,b]$ with $x(t)\in\M$ for all $t\in[a,b]$ satisfies
\begin{equation}\label{M hyp}
d_{\C}\bigl(x(b)\bigr)\leq e^{k(b-a)}d_{\C}\bigl(x(a)\bigr).
\end{equation}
Then for any trajectory $x(\cdot)$ of (DI)$_{G}$ that lies totally within $\M_i\cup\M$, we have that (\ref{SI estimate}) holds.
\end{lem}
\begin{rem}\label{M hyp ext}
{\rm As in the proof of Corollary~\ref{SI cor}, the inequality (\ref{M hyp}) remains valid if $x(a)$ and\slash or $x(b)$ are not in $\M$ provided $x(t)\in\M$ for $t\in(a,b)$.
}\end{rem}
\begin{rem}\label{explanation of Zeno}
{\rm The assumption (\ref{M hyp}) seems to almost equal the conclusion (\ref{SI estimate}).  Indeed, if $x(\cdot)$ is a trajectory so that $[a,b]$ can be partitioned in a manner $a=t_0< t_1 <\dots < t_{\ell} =b$ in which $x(\cdot)$ restricted to $(t_n,t_{n+1})$ belongs to either $\M$ or $\M_i$, then the conclusion follows immediately by applying either (\ref{M hyp}) or Corollary~\ref{SI cor} on each subinterval and concatenating the estimates.  In general, however, not all trajectories are like this, and one can even move in and out of $\M_i$ infinitely often (the so-called Zeno effect), or can reside in $\M_i$ for nontrivial time periods that are nowhere dense (see Example~\ref{Zeno}).  When this happens, there is no obvious procedure in how to paste the individual departures from $\M_i$ into one common estimate.  The following proof nonetheless finds a way around that by building \lq\lq impulsive arcs\rq\rq\ that uniformly approximate the original one and are well-behaved on the intervals of a partition.
}
\end{rem}

\begin{proof}  For notational convenience, assume $x(0)\in\M_i$ and $x(T)\in\M_i$ (by (\ref{M hyp}), there is no loss in generality in doing so).   Let $J:=\{t\in[0,T]:x(t)\in\M\}$, which is an open set (since $\overline\M_i\subseteq\M$) and so can be written
\[
J=\bigcup_{n=1}^{\infty} (a_n,b_n)
\]
where the intervals are pairwise disjoint.  The proof is much easier if the sum is finite (see Remark~\ref{explanation of Zeno}).  For fixed $\ell$, write
\[
J_{\ell} = \bigcup_{n=1}^{\ell} (a_n,b_n),
\]
which after reindexing can be assumed to satisfy
\[
b_0:=0\leq a_1<b_1\leq a_2 < b_2\leq \dots \leq a_{\ell} < b_{\ell}\leq T =: a_{\ell+1}.
\]
We note that for each $n=1,\cdots,\ell+1$, both $x(a_{n})$ and $x(b_{n-1})$ belong to $\M_i$.  Let
\[
r:=\inf_{\substack{t\in [0,T]\\ y\in\overline\M_i\backslash \M_i}}\|x(t)-y\|>0
\]
(this is strictly positive because the range of $x(\cdot)$ is compact and is disjoint from $\overline\M_i\backslash\M_i$), and choose $\ell$ sufficiently large so that
\begin{equation}\label{keep y close}
\meas(J\backslash J_{\ell})<\frac{r}{2\,e^{k\,T}\,\|G\|},
\end{equation}
where $\meas(I)$ denotes the Lebesgue measure of $I$ and $\|G\|$ is an upper bound of all the relevant velocities that can appear.

Note that $x(a_n)$ and $x(b_{n-1})$ belong to $\M_i$ for all $n=1,\dots,\ell+1$, and $x(t)\in \M$ for all $t\in (a_n,b_n)$. By assumption (\ref{M hyp}) and Remark~\ref{M hyp ext}, we have
\begin{equation}\label{first est}
d_{\C}\bigl(x(b_n)\bigr)\leq e^{k\,(b_n-a_n)}d_{\C}\bigl(x(a_{n})\bigr)\quad\forall\,n=1,\dots,\ell.
\end{equation}
We will build arcs using Proposition~\ref{Filippov} on the remaining intervals $[b_{n},a_{n+1}]$ that reside entirely within $\M_i$.  For $n = 0,\dots, \ell$, set $\varepsilon_n:=
\meas\bigl((b_{n},a_{n+1})\cap J\bigr)$.  Let $y_{n}(\cdot)$ be the arc $x(\cdot)$ that is restricted to the interval $[b_{n},a_{n+1}]$.  We calculate the closeness of $y_n(\cdot)$ to a trajectory of $(\M_i,F_i)$ by
\[
\rho_n:=\rho_{[b_{n},a_{n+1}]}\bigl(y_n(\cdot)\bigr)=
\int_{b_{n}}^{a_{n+1}} \dist\biggl(\dot x(s),F_i\bigl(y_{n}(s)\bigr)\biggr)\,ds \leq 2\,\|G\|\,\varepsilon_{n},
\]
where the final inequality is justified via Proposition~\ref{G=F solutions} (recall $y_n(t)= x(t)$ on $[b_n,a_{n+1}]$ and Proposition~\ref{G=F solutions} says $\dot x(t)\in F_i\bigl(x(t)\bigr)$ almost everywhere on the set where $x(t)\in \M_i$.)
By Proposition~\ref{Filippov}, there exists a trajectory $z_{n}(\cdot)$ of (DI)$_{F_i}$ defined on $[b_n,a_{n+1}]$ with $z_{n}(b_{n})=y_{n}(b_n)=x(b_n)$ and satisfying
\begin{equation}\label{second est}
\|y_{n}(a_{n+1})-z_{n}(a_{n+1})\|\leq  e^{k(a_{n+1}-b_{n})}\,\rho_n\leq 2\,\|G\|\,e^{k(a_{n+1}-b_{n})}\,\varepsilon_{n}.
\end{equation}
In the notation of Proposition~\ref{Filippov}, we have $\bar t=a_{n+1}$ because
\[
\rho_{[b_{n},t]}\bigl(y_n(\cdot)\bigr)\leq 2\,\|G\|\,\varepsilon \leq2\,\|G\|\,\meas(J\backslash J_{\ell})< r\,e^{-k\,T}
\]
by (\ref{keep y close}); see Remark~\ref{full Filippov}.  Since $z_{n}(\cdot)$ is a trajectory of $F_i$ remaining in $\M_i$, it satisfies the conditions of Corollary~\ref{SI cor} and so
\begin{equation}\label{third est}
d_{\C}\bigl(z_{n}(a_{n+1})\bigr)\leq e^{k\,(a_{n +1}-b_{n})} \,d_{\C}\bigl(z_{n}(b_{n})\bigr).
\end{equation}
The last ingredient needed is the trivial fact
\begin{equation}\label{fourth est}
d_{\C}\bigl(x(a_{n})\bigr)=d_{\C}\bigl(y_{n-1}(a_{n})\bigr)\leq d_{\C}\bigl(z_{n-1}(a_{n})\bigr) + \bigl\|y_{n-1}(a_{n})-z_{n-1}(a_{n})\bigr\|.
\end{equation}
We now calculate
\begin{eqnarray}
d_{\C}\bigl(x(T)\bigr)&=& d_{\C}\bigl(y_{\ell}(a_{\ell+1})\bigr)\nonumber
\\
&\leq& d_{\C}\bigl(z_{\ell}(a_{\ell+1})\bigr)
+\|y_{\ell}(a_{\ell +1})-z_{\ell}(a_{\ell+1})\|
\qquad\text{(by (\ref{fourth est}))}\nonumber\\
&\leq& e^{k\,(T-b_{\ell})}\biggl[ \,d_{\C}\bigl(z_{\ell}(b_{\ell})\bigr)+2\,\|G\|\,\varepsilon_{\ell}\biggr] \qquad\text{(by (\ref{third est}) and (\ref{second est})}  \nonumber\\
&=& e^{k\,(T-b_{\ell})}\biggl[ \,d_{\C}\bigl(x(b_{\ell})\bigr)+2\,\|G\|\,\varepsilon_{\ell}\biggr] \nonumber\\
&\leq& e^{k\,(T-a_{\ell})}\biggl[ \,d_{\C}\bigl(x(a_{\ell})\bigr)+2\,\|G\|\,\varepsilon_{\ell}\biggr] \label{first conclusion}
\end{eqnarray}
where (\ref{first conclusion}) follows from (\ref{first est}).  Repeating the previous argument on the interval $[a_{\ell-1},a_{\ell}]$ gives the similar inequality
\begin{equation*}
d_{\C}\bigl(x(a_{\ell})\bigr) \leq e^{k\,(a_{\ell}-a_{\ell-1})}\biggl[ \,d_{\C}\bigl(x(a_{\ell-1})\bigr)+2\,\|G\|\,\varepsilon_{\ell-1}
\biggr],
\end{equation*}
which after inserting into (\ref{first conclusion}) produces
\begin{equation*}
d_{\C}\bigl(x(T)\bigr) \leq e^{k\,(T-a_{\ell-1})}\biggl[ \,d_{\C}\bigl(x(a_{\ell-1})\bigr)+2\,\|G\|\,\bigl(\varepsilon_{\ell-1}+\varepsilon_{\ell}\bigr)\biggr].
\end{equation*}
After $\ell$ steps, the result is
\begin{equation}\label{second conclusion}
d_{\C}\bigl(x(T)\bigr) \leq e^{k\,T}\biggl[ \,d_{\C}\bigl(x\bigr)+2\,\|G\|\,\sum_{n=1}^{\ell}
\varepsilon_n \biggr].
\end{equation}
Finally, since $\sum_{n=1}^{\ell}
\varepsilon_n = \meas\bigl(J\backslash J_{\ell}\bigr)$, the claim (\ref{SI estimate}) follows from (\ref{second conclusion}) since $\meas\bigl(J\backslash J_{\ell}\bigr)\to 0$ as $\ell\to \infty$.
\end{proof}

We shall use Lemma~\ref{two subdomains} to finish the proof of Theorem~\ref{SI for stratified}.
We are given a trajectory $x(\cdot)$ of (DI)$_G$ defined on $[0,T]$, and must show (\ref{SI estimate}) holds.
Let $i$ be the unique index with $x(0)=x\in\M_i$.  The proof is now based on an induction argument with regard to the dimension $d_i$ of $\M_i$.
Assume first $d_i=N$.  Let $\widetilde T=\min\{T,\Esc\bigl(x(\cdot),\M_i,G)\}>0$, and recall this means $x(t)\in\M_i$ for $0\leq t<\widetilde T$ and either $\widetilde T=T$ or $x(\widetilde T)\in\bdry \M_i$.
The estimate (\ref{SI estimate}) holds with $T$ replaced by $\widetilde T$ by Corollary~\ref{SI cor}.  As an induction hypothesis, assume that for a dimension $d\leq N$, if $d_i\geq d$, then there exists $\widetilde T>0$ so that (\ref{SI estimate}) holds with $T$ replaced by any $\tilde t\in[0,\widetilde T]$.  We have just shown this property holds for $d=N$.  In order to show it holds for $d-1$, assume $d_i=d-1$ and let $\M$ consist of the union of all the subdomains $\M_j$ with $\M_i\cap\overline\M_j\not=\emptyset$.  The stratification assumptions imply $\M_i\subseteq\overline\M_j$ for each $j$ in this union, and thus $d_j\geq d_i+1\geq d$ for all such $j$.

We claim that $\M$ satisfies assumption (\ref{M hyp}) in Lemma~\ref{two subdomains}.  To see this, suppose $y(\cdot):[a,b]\to\M$ is a trajectory of (DI)$_G$.  Let
\[
\tilde b:=\sup\biggl\{t\in[a,b]:d_{\C}\bigl(y(t)\bigr)\leq e^{k(t-a)}d_{\C}\bigl(y(a)\bigr)\biggr\}
\]
Since $y(a)$ lies in a subdomain of dimension greater than or equal to $d$, the induction hypothesis says that $a<\tilde b$.  If $\tilde b<b$, then $y(\tilde b)\in\M$ and the induction hypothesis implies there exists $\widetilde T>0$ so that every $0\leq\tilde t\leq\widetilde T$ satisfies
\[
d_{\C}\bigl(y(\tilde b+\tilde t)\bigr)
\leq e^{k\,(\tilde t)}d_{\C}\bigl(y(\tilde b)\bigr)
\leq e^{k\,(\tilde b+\tilde t-a)}d_{\C}\bigl(y(a)\bigr).
\]
This contradicts that $\tilde b$ was a supremum, and thus $\M$ satisfies (\ref{M hyp}) as claimed.

Recall $x(0)\in\M_i$ and $d_i=d-1$.  Now $x(t)$  belongs to $\M_i\cup\M$ for all small $t>0$, say for $0\leq t\leq\widetilde T$.   This is because $x(\cdot)$ is continuous and $x(0)$ is a positive distance away from $\overline \M_i\backslash\M_i$ and every other subdomain that has dimension less than or equal to $d-1$.  Hence Lemma~\ref{two subdomains} can be applied on the interval $[0,\widetilde T]$, and the induction step is complete.

To finish the proof, we proceed in the same manner that we used above to verify that $\M$ satisfied (\ref{M hyp}).  Indeed, let
\[
\widetilde T:=\sup\biggl\{t\in[0,T]:d_{\C}\bigl(x(t)\bigr)\leq e^{k(t-a)}d_{\C}\bigl(x(0)\bigr)\biggr\}.
\]
One can show $\widetilde T$ equals $T$ in precisely the same way that $\tilde b=b$ was shown above.  The proof is now complete.
\end{proof}

\section{Weak invariance in stratified systems}\label{WI stratified}

Recall that a system $(\M,\Gamma)$ satisfying (SH) is weakly invariant {\it on} a closed set $\C$ {\it in} an open set $\U\subseteq\R^N$ provided for all $x\in \C\cap\M\cap\U$, there exists a trajectory $x(\cdot)$ of (DI)$_{\Gamma}$ on $\bigl[0,T\bigr)$ (with $T=\Esc(x(\cdot),\M\cap\U,\Gamma)$) so that $x(t)\in\C\cap\U$ for all $t\in[0,T)$.  It is known (cf. \cite{BOOK98}) that for $\M=\R^N$, this property is characterized by
\begin{equation}\label{G WI}
h_{\Gamma}(x,\zeta)\leq 0\quad\forall x\in\C\cap\U,\,\zeta\in\N_{\C}^P(x).
\end{equation}
By introducing $G^\sharp$ as a certain submultifunction of $G$, we were able to characterize strong invariance through a Hamilton-Jacobi inequality using $h_{G^\sharp}$ as the Hamiltonian.  Although (Obj1)$_{\ell}$ and (Obj2)$_{\ell}$ are thereby achieved (see Lemmas~\ref{lem1} and \ref{lem2} below), the first objective (Obj1)$_{u}$ can be maintained only if it is shown that the manner of reducing $G$ was not so stringent that it lost the ability to characterize weak invariance.  The following theorem provides the verification that it was not.

\begin{thm}\label{WI thm}  Suppose a closed set $\C\subseteq\R^N$, an open set $\U\subseteq\R^N$, and a stratified system are given.  Then $(\R^N,G)$ is weakly invariant on $\C$ in $\U$ if and only if
\begin{equation}\label{WI}
h_{G^\sharp}(x,\zeta)\leq 0\quad\forall x\in\C\cap\U,\,\zeta\in\N_{\C}^P(x).
\end{equation}
\end{thm}
\begin{proof}  Since $G^\sharp$ is contained in $G$, we have $h_G\leq h_{G^\sharp}$.  Thus if (\ref{WI}) is satisfied, then so is (\ref{G WI}) with $\Gamma=G$.  We can then conclude $(\R^N,G)$ is weakly invariant on $\C$ in $\U$.

Conversely, suppose $(\R^N,G)$ is weakly invariant on $\C$ in $\U$, and let $x\in\C\cap\U$ and $\zeta\in\N_{\C}^P(x)$ be given.  By weak invariance, there exists a solution $x(\cdot)$ to (DI)$_G$ that lies in $\C\cap\U$ for all $t\in[0,T)$.  Let $t_n\searrow 0$ so that $x_n:=x(t_n)$ satisfies $\frac{x_n-x}{t_n}\to v$ for some $v$.  We obviously have
\begin{equation}\label{v zeta <0}
\langle v,\zeta\rangle = \lim_{n\to\infty}\left\langle\frac{x_n-x}{t_n},\zeta\right\rangle \leq \frac{1}{2\sigma\,t_n}\|x_n-x\|^2\to 0
\end{equation}
as $n\to \infty$.  The convergence to $0$ is a consequence of the same inequality that was recorded in (\ref{bound on x(t)-x}).  We claim that
\begin{equation}\label{co G sharp}
v\in\co G^\sharp(x),
\end{equation}
where $\co G^\sharp(x)$ denotes the set operation of taking the convex hull of $G^\sharp(x)$.  Along with (\ref{v zeta <0}), this will imply (\ref{WI}) since one has
\begin{equation*}
h_{G^\sharp}(x,\zeta)=\inf_{v'\in G^\sharp(x)}\langle v',\zeta\rangle = \inf_{v'\in \co G^\sharp(x)}\langle v',\zeta\rangle.
\end{equation*}

We seek to verify (\ref{co G sharp}).
 Let $J_i^n:=\{t\in[0,t_n]:x(t)\in\M_i\}$, and write $\bar i(x)$ for those indices for which $\mu^n_i:=\meas(J^n_i)>0$ for all $n$.  Then $i\in\bar i(x)$ implies $x\in\overline\M_i$.
Without loss of generality, we may assume for each $i\in\bar i(x)$ there exists $0\leq\lambda_i\leq 1$ and $v_i\in\R^N$ so that
\begin{equation}\label{convergences}
\frac{\mu^n_i}{t_n}\to \lambda_i,\;\sum_{i\in\bar i(x)}\lambda_i=1,
\quad\text{and}\quad
\frac{1}{\mu_i^n}\int_{J_i^n}\dot x(t)\,dt\to v_i\in \overline F_i(x)
\end{equation}
as $n\to\infty$.
The last inclusion holds because
\begin{eqnarray*}
v_i&=&\lim_{n\to\infty}\frac{1}{\mu_i^n}\int_{J_i^n}\dot x(t)\,dt \\
&\in &
\lim_{n\to\infty}\left[\frac{1}{\mu_i^n}\int_{J_i^n}\proj\bigl(\dot x(t),\overline F_i(x)\big)\,dt +
\frac{1}{\mu_i^n}\int_{J_i^n}k\|x(t)-x\|\B\,dt\right] \\
&\subseteq &
\lim_{n\to\infty}\left[\overline F_i(x) +
k\|G\|\,\left[\frac{1}{\mu_i^n}\int_{J_i^n} t\,dt\right]\B \right] = \overline F_i(x).
\end{eqnarray*}
We now have
\begin{eqnarray*}
v&=&\lim_{n\to\infty}\frac{x_n-x}{t_n} = \lim_{n\to\infty}\frac{1}{t_n}\int_{0}^{t_n}\dot x(t)\,dt  \\
&=& \sum_{i\in\bar i(x)}\lim_{n\to\infty}
\frac{\mu^n_i}{t_n}\left[\frac{1}{\mu_i^n}\int_{J^n_i}\dot x(t)\,dt \right]\\
&=& \sum_{i\in\bar i(x)}\lambda_i v_i \in
\sum_{i\in\bar i(x)}\lambda_i \overline F_i(x)
\end{eqnarray*}
by (\ref{convergences}).  Hence $v\in G(x)$ by (\ref{G rep}).  Now let $\M:=\cup_{i\in\bar i(x)}\M_i$, and since $x_n\in\M$ for all large $n$, we must also have $v\in\T_{\overline\M}(x)$.  The proof is completed by verifying
\begin{equation}\label{last inclusion}
\left\{\co\bigcup_{i\in\bar i(x)}
\overline F_i(x)\right\}
\bigcap \T_{\overline\M}(x)
\;\subseteq\; G(x)\bigcap \T_{\overline\M}(x)\;\subseteq\; \co G^{\sharp}(x).
\end{equation}
Indeed, since $F_i(y)\subseteq\T_{\M_i}(y)$ whenever $y\in\M_i$, one must have
\[
\overline F_i(x)\bigcap \T_{\overline\M}(x)=\overline F_i(x)\bigcap \T_{\overline\M_i}(x)
\]
whenever $x\in\overline\M_i$.  Taking the union over $i\in\bar i(x)$ and then the convex hull yields (\ref{last inclusion}), and finishes the proof.
\end{proof}

\begin{rem}\label{remark on proof}
{\rm It may reasonably be surmised that the vectors $v_i$ in the above proof belong to $\overline F_i(x)\cap \T_{\overline\M_i}(x)$.  However this is not the case in general, further demonstrating how complicated trajectories to stratified systems can be.  See Example~\ref{EX3}.
}
\end{rem}
\begin{rem}
{\rm The proof of Theorem~\ref{WI thm} actually shows that
\begin{equation}\label{infinitesimal generator}
\underset{t\searrow 0}{\Klim}\frac{R^{(t)}(x)-x}{t}\subseteq \co G^{\sharp}(x),
\end{equation}
where $\Klim$ denotes the Kuratowski limit of sets and $R^{(t)}(x)$ is the reachable set at time $t$ from $x$ (The definition of $\Klim$ is provided in the appendix).  Such a result for Lipschitz dynamics is relatively simple and can be used to characterize the reachable set semigroup in terms of its \lq\lq infinitesimal generator\rq\rq.  See \cite{W90}.  The observation (\ref{infinitesimal generator}) gives further justification for calling $G^{\sharp}$ the essential velocity multifunction.
}
\end{rem}

\section{Examples}\label{sec:Examples}  We offer a pair of examples here to illustrate (1) why the proof of Lemma~\ref{two subdomains} has to be so complicated, and (2) a curious phenomenon that arose in the proof of Theorem~\ref{WI thm} when the structural condition needed to be invoked on the velocity sets.

\begin{exam}\label{Zeno}
{\rm We provide a simple technique with $N=2$ for which a trajectory can be constructed to exhibit very complicated behavior while crossing a manifold.  Let $\M_1$ and $\M_2$ be the upper and lower half planes respectively and $\M_3$ the $x_1$-axis.  Let the dynamics be given by
\begin{eqnarray*}
F_1\left(\begin{matrix}x_1\\x_2\end{matrix}\right)&=&F_2\left(\begin{matrix}x_1\\x_2\end{matrix}\right)=\left\{\left(\begin{matrix}1\\ u\end{matrix}\right): -1\leq u \leq 1\right\};\\
F_3\left(\begin{matrix}x_1\\0\end{matrix}\right)&=&\left\{\left(\begin{matrix}u\\ 0\end{matrix}\right):-1\leq u\leq 1\right\}
\end{eqnarray*}
For an interval $[a,b]$, define $\varphi_{a,b}:[a,b]\to\R^2$ by
\[
\varphi_{a,b}(t)= \left\{
\begin{array}{cl}
1 &\quad\text{if }a\leq t\leq\frac{3a+b}{4} \\
-1&\quad\text{if }\frac{3a+b}{4}\leq t\leq \frac{a+3b}{4} \\
1 &\quad\text{if }\frac{a+3b}{4}\leq t\leq b
\end{array}\right.
\]
and a trajectory $x(\cdot)$ by
\[
x(t)= \int_a^t \left(\begin{matrix}1\\ \varphi_{a,b}(s)\end{matrix}\right)\, ds.
\]

\begin{figure}[h]
\centering
\begin{tikzpicture}
\coordinate (P1) at (-2,0);
\coordinate (P2) at (-1,1);
\coordinate (P3) at (1,-1);
\coordinate (P4) at (2,0);

\draw[thick] (P1) -- (P2) -- (P3) -- (P4);
\draw[dashed, thin] (-3.5,-1.5) -- (-3.5,1.5) -- (2.5,1.5) -- (2.5,-1.5) -- cycle;
\draw[<->,thin] (-3.4,0) -- (2.4,0);

\fill (P1) circle (2pt);
\fill (P4) circle (2pt);
\fill (0,0) circle (2pt);

\path (-2,0) node [below] {\footnotesize $t=a$};
\path (2,0) node [below] {\footnotesize $t=b$};
\path (0,0) node [below] {\footnotesize $t=\frac{a+b}{2}$};

\path (1,1) node {\footnotesize $\M_1$};
\path (-1,-1) node {\footnotesize $\M_2$};
\path (-3,0) node {\footnotesize $\M_3$};

\end{tikzpicture}\\
\caption{The trajectory $x(\cdot)$}
\label{Zeno1}
\end{figure}
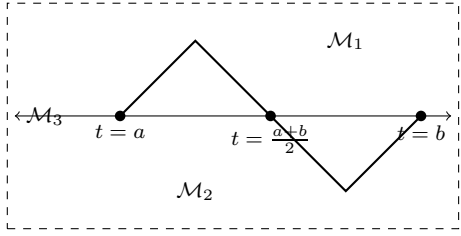
Then $x(\cdot)$ is a trajectory on $[a,b]$ that begins and ends on $\M_3$ with half of its time spent in $\M_1$ and the other half in $\M_2$.  See Figure~\ref{Zeno1}.

Now suppose $\bigl\{(a_n,b_n)\bigr\}$ is {\em any} collection of pairwise disjoint finite intervals in $(0,+\infty)$, and let
\[
v(t)=\left\{
\begin{array}{cl}
\left(\begin{matrix}1\\ \varphi_{a_n,b_n}(t)\end{matrix}\right) &\quad\text{ if }t\in (a_n,b_n) \\
\left(\begin{matrix}1\\ 0\end{matrix}\right)
&\quad \text{ if }t\notin {\bd\bigcup_{n=1}^{\infty}(a_n,b_n)\ed}.
\end{array}\right.
\]
The trajectory $x(t)=\int_0^t v(s)\,ds$ crosses the $x$-axis at the midpoint of every interval $(a_n,b_n)$ and lies on the $x$-axis off of $\cup_{n=1}^{\infty}(a_n,b_n)$.  This trajectory will exhibit Zeno behavior if $b_n\to 0$, and will be very complicated indeed if, for example, the set $A:=(0,+\infty)\backslash \cup_{n=1}^{\infty}(a_n,b_n)$ is nowhere dense with $\meas\bigl((t_0,t_1)\cap A\bigr)>0$ for all $0\leq t_0<t_1$.
}
\end{exam}

\begin{exam}\label{EX3}
{\rm
We modify the previous example by splitting all the manifolds by inserting the $x_2$-axis.  We need not label or introduce additional dynamics because we are only interested here on behavior when $x_1>0$, and refer to $\M_1,\,\M_2,\,\M_3$ as being the same as in Example~\ref{Zeno} intersected with the right half space.  For an interval $(a,b)$, now define
\[
\varphi_{a,b}(t)= \left\{
\begin{array}{cl}
1 &\quad\text{if }a\leq t\leq\frac{4a+b}{5} \\
-1&\quad\text{if }\frac{4a+b}{5}\leq t\leq \frac{2a+3b}{5} \\
1 &\quad\text{if }\frac{2a+3b}{5}\leq t\leq\frac{a+4b}{5}
\end{array}\right.
\]
Let $\{c_n\}$ be any sequence strictly decreasing to $0$ and set $a_n=c_n$ and $b_n=c_{n+1}$.  Let
\[
v(t)=\left\{
\begin{array}{cl}
\left(\begin{matrix}1\\ \varphi_{a_n,b_n}(t)\end{matrix}\right) &\quad\text{ if }t\in \left(a_n,\frac{a_n+4b_n}{5}\right) \\
\left(\begin{matrix}-1\\ 0\end{matrix}\right)
&\quad \text{ if }t\notin {\bd\bigcup_{n=1}^{\infty}\left(a_n,\frac{a_n+4b_n}{5}\right)\ed}.
\end{array}\right.
\]
The arc defined by $x(t)=\int_{0}^{t} v(s)\,ds$
is another Zeno-type arc, and has derivative equal to $-1$ while on the $x_1$-axis.  This occurs one-fifth of the time on every interval $(0,c_n)$, and shows that the vector $v_i$ produced in (\ref{convergences}) may not lie in $\T_{\overline\M}(0)$, where in this case $\M=\M_1\cup\M_2\cup\M_3$.
}\end{exam}

\section{Conclusion}\label{sec:Conclusion}

Our main result was a characterization of flow invariance for dynamics defined on a stratified system. With the underlying dynamics now discontinuous, the classical theory involving Hamiltonian inequalities required a modification.  We imposed two conditions (proximal smoothness and wedgeness) on the subdomains that were not considered in the original formulation in \cite{BH07} of a stratified system.  We believe these conditions are natural and may become standard in future studies in optimal control where a switch to a lower dimensional manifold demands greater attention to detail.
In the present context, these conditions were helpful in identifying what we called $G^\sharp$, the essential velocity multifunction.
Before fully investigating the essential velocity multifunction, we first gave a new general theorem for the sufficiency of strong invariance for a system with perhaps non-Lipschitz data but satisfying a potentially weaker assumption that we called the Euler arc property.  It is not clear which multifunctions beyond the Lipschitz ones satisfy this property.  Another difficult feature of the analysis was to show $h_{G^{\sharp}}$ was also capable of characterizing weak invariance.

The general approach of utilizing the stratified structure of this paper is being applied to analyze other classical problems.  In particular, a new approach to state constraint problems will appear in \cite{WZ11} and to reflected problems in \cite{SWZ11}.  The basic idea is to identify the \lq\lq essential velocity\rq\rq\ multifunction $G^{\sharp}$ in these problems whose associated Hamiltonian achieves the following:
\begin{itemize}
\item[(Obj1)$_{u}$]  $G^{\sharp}$ must be large enough so that an Hamilton-Jacobi inequality of type $h_{G^{\sharp}}\leq 0$ will characterize weak invariance, and
\item[(Obj2)$_{\ell}$]  $G^{\sharp}$ cannot be too large so that an Hamilton-Jacobi inequality of type $h_{G^{\sharp}}\geq 0$ will characterize strong invariance.
\end{itemize}
Finding the right balance between these competing objectives leads respectively to existence (with (Obj1)) and uniqueness (with (Obj2)) results to Hamilton-Jacobi equations.

\appendix
\section{}

This appendix delves into the details of the \lq\lq relative wedge\rq\rq\ assumption, culminating in the proof of Lemma~\ref{rb of tangent cone}.  A wedge assumption is that every normal cone at a point of the boundary is pointed, so our remarks will focus on \lq\lq pointedness\rq\rq\ rather than wedgeness per se.  This is the case for the relative concept as well.

\subsection{Pointedness}
A closed cone $K\subseteq\R^N$ is said to be pointed if $K\cap(-K)=\{0\}$.  A result by Rockafellar \cite{Rock79} says that a closed set $\C\subseteq\R^N$ has a pointed (Clarke) normal cone at $c\in\C$ if and only if $\C$ is epi-Lipschitz around $c$.  The latter means that after a unitary change of coordinates, $\C$ locally has the form of an epigraph of a Lipschitz function.  Technically, this means there exists a unitary map $\Phi:\R^N\to\R^{N-1}\times\R$, a constant $\varepsilon>0$, and a Lipschitz function $g:\varepsilon\,\B_{N-1}\to\R$ with $g(0)=0$  so that
\begin{equation}\label{epi g}
\Phi(\C-c)\bigcap \biggl\{\varepsilon\,\B_{N-1}\times(-\varepsilon,\varepsilon)\biggr\}=\bigl\{\epi g\bigr\} \bigcap \biggl\{\varepsilon\,\B_{N-1}\times(-\varepsilon,\varepsilon)\biggr\}.
\end{equation}
($\B_{N-1}$ denotes the unit ball in $\R^{N-1}$).  The condition $\N_{\C}(c)$ being pointed is also equivalent to $\T_{\C}(c)$ having nonempty interior (see \cite{Rock79}).

In addition to $\N_{\C}(c)$ being pointed, suppose that $\C$ is proximally smooth.  Then the Lipschitz function $g$ in (\ref{epi g}) is {\it lower $C^2$} (\cite{RW98,Rock82} and \cite[Theorem~5.2]{CSW95}).  This means there exists a representation
\begin{equation}\label{lower C^2}
g(y)=\max_{a\in A}\bigl\{g_a(y):a\in A\bigr\},
\end{equation}
where $A$ is a compact metric space and $g_a:\varepsilon\,\B_{N-1}\to \R$ is such that $g_a(\cdot)$, $\nabla\,g_a(\cdot)$, and $\nabla^2\, g_a(\cdot)$ are all continuous as functions over $(y,a)\in \varepsilon\,\B_{N-1}\times A$.  For all $y\in\varepsilon\,\B_{N-1}$, it can be shown that (see \cite{Rock82,RW98}) $g(\cdot)$ as in (\ref{lower C^2}) has a directional derivative $g'(y;w)$ in all directions $w\in\R^{N-1}$, and moreover, satisfies Danskin's formula
\begin{equation}\label{Danskin}
g'(y;w)=\max_{a\in A(y)}\langle \nabla_y\,g_a(y),w\rangle,
\end{equation}
where $A(y):=\{a\in A:g(y)=g_a(y)\}$ (see also \cite{FHC75}).

One advantage of having the representation (\ref{lower C^2}) is that the boundary of $\C$ near $x$ can be easily recognized as the points $\{\Phi^{-1}\bigl(y,g(y)\bigr):\|y\|<\varepsilon\}$.  This fact will be used below in the proof Lemma~\ref{rb of tangent cone}.

The representation (\ref{lower C^2}) provides a relatively simple means to calculate the normal and tangent cones as well as determine their interiors and boundaries.  Recall that $\Phi$ was unitary, which implies that normal and tangent cones are preserved under the transformation.  Denote by $\cone K$ the convex cone generated by $K$.  One can verify that
\begin{eqnarray}
\N_{\C}(c)&=&\cone \Phi^{-1}\biggl[\bigl\{(\nabla\, g_a(0),-1):a\in A(0)\bigr\}\biggr]  \nonumber\\
&=& \bigcap_{\nu>0} \co \bigcup_{\|c'-c\|\leq \nu} \N_{\C}(c')\label{normal formula} \\
\T_{\C}(c)&=& \cone \Phi^{-1}\biggl[\bigl\{(w,1):\langle\nabla\, g_a(0),w\rangle\leq 1\quad\forall\,a\in A(0)\bigr\}\biggr] \nonumber\\
&=& \underset{\nu\to 0}{\Klim}\bigcap_{\|c'-c\|\leq \nu} \T_{\C}(c').\label{tangent formula}
\end{eqnarray}
The proofs of the second representations (\ref{normal formula}) and (\ref{tangent formula}) rely heavily on the regularity assumptions of $(y,a)\mapsto g_a(y),\,\nabla g_a(y)$.  Here, $\K=\underset{\nu\to 0}{\Klim}\,\K_\nu$ refers to the Kuratowski limit of sets, which means two things:
\[
\liminf_{\nu\to 0}\dist(x,\K_{\nu})=0 \quad\forall\, x\in\K,
\]
and, secondly, if $\nu_n\to 0$ and $x_{\nu_n}\in\K_{\nu_n}$ are sequences with $x_{\nu_n}\to x$, then $x\in\K$.

One can derive formulas from (\ref{tangent formula}) for the interior and boundary of the tangent cone.  We leave these routine derivations to the reader.
\begin{eqnarray}
\iint\T_{\C}(c)&=& \Phi^{-1}\biggl[\cone\bigl\{(w,r):\exists\mu>0 \text{ so that }\nonumber\\
&&\qquad \langle\nabla\, g_a(0),w\rangle\leq \|w\|\bigl[r-\mu(\|\nabla\, g_a(0)\|+1)\bigr]\;\forall\,a\in A(0)\bigr\}\biggr] \nonumber   \\
   &=&  \biggl\{v\in\R^N: \exists \mu>0 \text{ so that }\langle \zeta,v\rangle \leq -\mu\|\zeta\|\;\forall \zeta\in\N_{\C}(c)\biggr\}; \label{int of tangent cone}\\
\bdry\T_{\C}(c)&=& \Phi^{-1}\cone\bigl\{(w,-1):\exists a\in A(0)\text{ so that }g'(0;w)=\langle \nabla\,g_a(0),w\rangle\bigr\}  \nonumber \\
&=&\biggl\{v\in\T_{\C}(c): \exists \zeta\in\N_{\C}(x)\backslash\{0\} \text{ so that }\langle \zeta,v\rangle =0\biggr\}  \label{bdry of tangent}.
\end{eqnarray}

\subsection{Relative concepts}
In this subsection, we extend the pointed concept of a normal cone to lower dimensional manifolds.  We mention that we are not introducing this concept for arbitrary proximally smooth sets, but rather only to those that are closures of embedded manifolds.

Suppose $\M$ is a $d$-dimensional manifold embedded in $\R^N$, and its closure $\overline\M$ is proximally smooth.  If $d<N$ and $x\in\overline\M$, then by considering a local chart near $x$, one can show there must exist a {\it subspace} of dimension at least $N-d$ belonging to $\N_{\overline\M}(x)$.
We say that $\N_{\overline\M}(x)$ is {\it relatively pointed} provided it contains no subspace of dimension larger than $N-d$.  This assumption allows for a description of the normal and tangent cones associated with lower dimensional manifolds in a manner similar to what was described above for dimension $N$.  After a change in coordinates, the normal cone is of the form $\N^d\times \R^{N-d}$ and the tangent cone of form $\T^d\times\{0_{N-d}\}$ where $\N^d$, $\T^d$ are convex cones lying in the subspace of $\R^N$ consisting of the first $d$-coordinates, and $0_{N-d}$ is the zero vector with $N-d$ coordinates.  We have that $\N_{\overline\M}(x)$ is {relatively pointed} provided $\N^d$ is pointed.  This is equivalent to saying that the cone $\T^d$ has nonempty interior, or again equivalently, that the relative interior (in the sense of convex analysis - see \cite{RTRconvex70}) of the cone $\T_{\overline\M}(x)$ is of dimension $d$.  The relative interior of a convex set $C$ is denoted by $\ri\C$, and its relative boundary $\rb\C$ is defined as $\overline\C\backslash\ri\C$.

We return to the situation where a (proximally smooth) stratification $\{\M_1,\dots,\\ \M_M\}$ is given, and recall the dimension of $\M_i$ is $d_i$.  Fix $i$, and suppose $x\in\overline{\M}_i$.  If $x\in\M_i$ then $\ri \T_{\overline\M_i}(x)=\T_{\M_i}(x)\simeq \R^{d_i}$ and $\rb \T_{\overline\M_i}(x)=\emptyset$.  An important property of $v$ when $x\in\overline\M_i\backslash\M_i$ and $v\in\rb \T_{\overline\M_i}(x)$ was stated in Lemma~\ref{rb of tangent cone}, whose proof we can finally provide.

\begin{proof}[Proof of Lemma~\ref{rb of tangent cone}]

We are given $x\in\overline\M_i$ and $v\in \rb \T_{\overline\M_i}(x)$.  We must show there exists another subdomain $\M_j\subseteq\overline\M_i$ with $x\in\overline\M_j$ and $v\in\T_{\overline\M_j}(x)$.  Note that the relative boundary $\rb \T_{\overline{\M}_i}(x)$ can be not empty only if $d_i>0$, so without loss of generality, we can assume $d_i=N$.  By (\ref{defn of tangent}), there exists a sequence $t_n\searrow 0$ so that $\frac{x+t_nv-x_n}{t_n}\to 0$ where $x_n=\proj_{\overline M_i}(x+t_nv)$ (the closest point is unique for large $n$ by proximal smoothness).  The rest of the proof is simple if $x+t_nv\notin \M_i$ for infinitely many $n$.  Indeed, in that case $x_n\in\bdry\M_i$, and since there are only finite many subdomains in the stratification, there exists an index $j$ so that $x_n\in\M_j$ for infinitely many $n$.  It is clear then that $x\in\overline\M_j$ and $v\in\T_{\overline\M_j}(x)$.  The other case is where $x_n=x+t_nv\in\M_i$ for all large $n$.  The proof here is considerably more difficult and we resort to the representation (\ref{lower C^2}) that is available for $\overline\M_i$.  Changing coordinates, we can assume without loss of generality that $x=0$ and that $\overline\M_i=\epi g$ (locally) where $g$ is of the form (\ref{lower C^2}).  For $z\in\R^N$, we write $z=(z^{N-1},z^{N})\in\R^{N-1}\times \R$, and with this notation, are given $x_n=t_n v$ satisfying
\begin{equation}\label{convergence}
x_n^{N-1}= {t_n}\,v^{N-1}\quad \text{and}\quad \frac{x_n^{N}}{t_n}\to v^N.
\end{equation}
The assumption $v\in\bdry\T_{\overline\M_i}(x)$ implies $v^N=g'(0;v^{N-1})$ by (\ref{bdry of tangent}).  From Danskin's formula (\ref{Danskin}), there exists $a\in A(0)$ for which
\begin{equation*}\label{dir derivative attained}
v^N=g'(0;v^{N-1})=\langle \nabla g_a(0),v^{N-1}\rangle.
\end{equation*}
We calculate using this that
\begin{eqnarray}
v^{N}&=&\langle \nabla g_a(0),v^{N-1}\rangle= \lim_{n\to\infty} \frac{g_a\bigl(t_n v^{N-1}\bigr)}{t_n}\nonumber\\
&\leq& \lim_{n\to\infty} \frac{g\bigl(t_n v_n^{N-1}\bigr)}{t_n}\nonumber\\
&\leq& \lim_{n\to\infty} \frac{x_n^N}{t_n}= v^N.
 \label{maxed} 
\end{eqnarray}
where the inequality in (\ref{maxed}) follows since we are assuming $x_n\in\M_i$ (and thus $x_n^{N}>g\left(x_n^{N-1}\right)$), and equality just restates (\ref{convergence}).  Since $\bigl(x_n^N,g(x_n^N)\bigr)\in\bdry\overline\M_i$ for each $n$, there exists an index $j$ so that $\M_j$ contains infinitely many of them.  Therefore $0\in\overline\M_j$ and $v\in\T_{\overline\M_j}(0)$.  Since $\M_j\subseteq\overline\M_i$ and $\M_j\cap\M_i=\emptyset$, we must have $d_j<d_i$, and the proof is complete.
\end{proof}

\bibliographystyle{plain}
\bibliography{newreferences}

\end{document}